\documentclass{article}
\usepackage[utf8]{inputenc}
\usepackage[english]{babel}
\usepackage{amsmath,amssymb,epsfig,amsthm}
\usepackage{graphicx}
\usepackage{comment}
\usepackage{array}
\usepackage{algorithm,algpseudocode}
\usepackage{enumerate}
\usepackage{algpseudocode}
\usepackage{xurl}
\usepackage{xargs}     
\usepackage{subcaption}
\usepackage{mwe}
\usepackage{appendix}
\usepackage[margin=1in]{geometry} 
\usepackage{xcolor}

\newcommandx{\unsure}[2][1=]{\todo[linecolor=red,backgroundcolor=red!25,bordercolor=red,#1]{#2}}
\newcommandx{\change}[2][1=]{\todo[linecolor=blue,backgroundcolor=blue!25,bordercolor=blue,#1]{#2}}
\newcommandx{\info}[2][1=]{\todo[linecolor=OliveGreen,backgroundcolor=OliveGreen!25,bordercolor=OliveGreen,#1]{#2}}
\newcommandx{\improvement}[2][1=]{\todo[linecolor=Plum,backgroundcolor=Plum!25,bordercolor=Plum,#1]{#2}}
\newcommandx{\thiswillnotshow}[2][1=]{\todo[disable,#1]{#2}}

\newcommand\numberthis{\addtocounter{equation}{1}\tag{\theequation}}
\allowdisplaybreaks[1]
\newtheorem{theorem}{Theorem}

\newtheorem{assumption}{Assumption}

\newtheorem{condition}{Condition}

\newtheorem{corollary}{Corollary}

\newtheorem{lemma}{Lemma}

\newtheorem{proposition}{Proposition}
\newtheorem{remark}{Remark}

\numberwithin{equation}{section}

\DeclareMathOperator{\diag}{diag}

\newcommand{\calB}{\ensuremath{\mathcal{B}}}
\newcommand{\calC}{\ensuremath{\mathcal{C}}}

\newcommand{\calH}{\ensuremath{\mathcal{H}}}

\newcommand{\calG}{\ensuremath{\mathcal{G}}}

\newcommand{\calP}{\ensuremath{\mathcal{P}}}

\newcommand{\calS}{\ensuremath{\mathcal{S}}}
\newcommand{\calM}{\ensuremath{\mathcal{M}}}
\newcommand{\calN}{\ensuremath{\mathcal{N}}}

\newcommand{\calE}{\ensuremath{\mathcal{E}}}
\newcommand{\calJ}{\ensuremath{\mathcal{J}}}

\newcommand{\norm}[1]{\left\|{#1}\right\|}
\newcommand{\abs}[1]{\left|{#1}\right|}
\newcommand{\ceil}[1]{\lceil{#1}\rceil}

\newcommand{\set}[1]{\left\{{#1}\right\}}

\newcommand{\expec}{\ensuremath{\mathbb{E}}}
\newcommand{\matR}{\ensuremath{\mathbb{R}}}

\newcommand{\prob}{\ensuremath{\mathbb{P}}}

\newcommand{\lt}{l(z^{(t)},z)}

\definecolor{asparagus}{rgb}{0.53, 0.66, 0.42}

\newcommand{\indic}{\ensuremath{\mathbf{1}}} 

\newcommand{\R}{\ensuremath{\mathbb{R}}}

\newcommand{\gpm}{\texttt{GPM}}
\newcommand{\spec}{\texttt{Spec}}
\newcommand{\hl}{\texttt{HL}}

\newcommand{\ov}{\Omega_{i,k}(\delta)}
\newcommand{\zz}{z'}

\newcommand{\rev}[1]{\textcolor{black}{#1}}

\usepackage[round]{natbib}
\begin{document}

\title{Minimax Optimal Clustering of Bipartite Graphs with a Generalized Power Method}
\author{Guillaume Braun\footnotemark[1]\\ \texttt{guillaume.braun@inria.fr} \and Hemant Tyagi\footnotemark[1]\\  \texttt{hemant.tyagi@inria.fr}}
\date{\today}

\renewcommand{\thefootnote}{\fnsymbol{footnote}}
\footnotetext[1]{Inria, Univ. Lille, CNRS, UMR 8524 - Laboratoire Paul Painlev\'{e}, F-59000 }

\maketitle

\begin{abstract}
    Clustering bipartite graphs is a fundamental task in network analysis. In the high-dimensional regime where the number of rows $n_1$ and the number of columns $n_2$ of the associated adjacency matrix are of different order, existing methods derived from the ones used for symmetric graphs can come with sub-optimal guarantees. Due to increasing number of applications for bipartite graphs in the high dimensional regime, it is of fundamental importance to design optimal algorithms for this setting. The recent work of \cite{Ndaoud2021ImprovedCA} improves the existing upper-bound for the misclustering rate in the special case where the columns (resp. rows) can  be partitioned into $L = 2$ (resp. $K = 2$) communities. Unfortunately, their algorithm cannot be extended to the more general setting where $K \neq L \geq 2$. We overcome this limitation by introducing a new algorithm based on the power method. We derive conditions for exact recovery in the general setting where $K \neq L \geq 2$, and show that it recovers the result in \cite{Ndaoud2021ImprovedCA}. \rev{We also derive a minimax lower bound on the misclustering error when $K=L$ under a symmetric version of our model, which matches the corresponding upper bound up to a factor depending on $K$}.
\end{abstract}

\section{Introduction}

The interactions between objects of two different types can be naturally encoded as a bipartite graph where nodes correspond to objects and edges to the links between the objects of different type. One can find examples of such data in various fields, e.g., interactions between customers and products in e-commerce \citep{customer-product}, interactions between plants and pollinators \citep{plants-pollinators}, investors and assets networks \citep{financial-bipartite}, judges vote predictions \citep{justice} and constraint satisfaction problems \citep{csp}.

Clustering is one of the most important analysis tasks on bipartite graphs aimed at gathering nodes that have similar connectivity profiles. To this end, several methods  have been proposed in the literature, e.g., convex optimization  approaches \citep{convex-bi}, spectral methods \citep{spectral-bi}, modularity function maximization \citep{Beckett2016ImprovedCD},  pseudo-likelihood \citep{optimal-bi} and variational approaches \citep{Keribin2015EstimationAS}. The performance of the algorithms are generally evaluated under the Bipartite Stochastic Block Model (BiSBM), a variant of the Stochastic Block Model (SBM), where the partitions of the rows and the columns are decoupled. In particular, edges are independent Bernoulli random variables with parameters depending only on the communities of the nodes.

When the number of rows $n_1$ (corresponding to type I objects) and the number of columns $n_2$ (corresponding to type II objects) of the adjacency matrix associated to a bipartite graph are of the same order, the BiSBM behaves similarly to the SBM. However, in the high dimensional setting where $n_2 \gg n_1$, classical methods that are applicable when $n_1$ is of the same order as $n_2$ can fail. In particular, when the bipartite graph is very sparse, it becomes impossible to consistently estimate the latent partition of the columns, whereas it is still possible to estimate the latent partition of the rows. Hence, methods based on estimating the latent partitions of both rows and columns will necessarily fail in the high dimensional regime when the bipartite graph is very sparse. However, the high dimensional regime appears in many applications, e.g. hypergraphs where the number of columns corresponds to the number of hyperedges, or in e-commerce where the number of customers could be much smaller than the number of products (or vice-versa). Hence it is important to understand how to design statistically optimal algorithms in this regime.

Recently, the work of \cite{Ndaoud2021ImprovedCA} improved the state-of-the-art conditions for exact recovery of the latent partition of the rows under the BiSBM when $n_2 \gtrsim n_1 \log n_1$. Unfortunately, their method, which can be understood as a generalized power method, uses a centering argument that only works in the special case where there are $K=2$ latent communities for the rows and $L=2$ latent communities for the columns. Moreover, they only consider a  Symmetric BiSBM (SBiSBM) where the edge probabilities can take two values (see Section \ref{sec:stat_frame}). It is not clear how their method can be extended to the more general setting where $K \neq L \geq 2$, and with more general connectivity matrices. To overcome this limitation, we propose a new algorithm -- also based on the generalized power method -- that can be applied to general BiSBMs, and has similar theoretical guarantees when specialized to the setting of \cite{Ndaoud2021ImprovedCA}.

\subsection{Main contributions} Our results can be summarized as follows.
\begin{itemize}
    \item We present a novel iterative clustering method that can be applied to general BiSBMs, unlike the one proposed by \cite{Ndaoud2021ImprovedCA}. We analyze our algorithm under the BiSBM without restrictions on $K$ and $L$, and derive an upper bound on the misclustering error. In particular, we show that our algorithm achieves exact recovery when the sparsity level $p_{max}$ of the graph satisfies $p_{max}^2 = \Omega(\frac{\log n_1}{n_1n_2})$ for fixed $K,L$. This is the same sparsity regime obtained in  \cite{Ndaoud2021ImprovedCA} for the SBiSBM with $K = L = 2$. We remark that our bounds are non-asymptotic and showcase the full dependence on $K,L$. 
    
    \item We derive a minimax lower bound for the misclustering error in the special case of an SBiSBM with $L=K=2$ that matches the corresponding upper bound of our algorithm, and is the first such lower bound for this problem. It completes the work of \cite{Ndaoud2021ImprovedCA} which only shows that an oracle version of their algorithm fails to achieve exact recovery when  $p_{max}^2\leq \epsilon \frac{\log n_1}{n_1n_2}$ where $\epsilon$ is a small enough constant. \rev{We also show that the above analysis extends to the case $K = L \geq 2$, and that the ensuing lower bound matches the upper bound up to a factor depending on $K$.}
    
    \item As part of our analysis, we derive a concentration inequality for matrices with independent centered Binomial entries (see part $3$ of Lemma \ref{lem:conc_mat}) that could be of independent interest.
    
\end{itemize}

 Our findings are complemented by numerical experiments on synthetic data.

\subsection{Related work}
A clustering strategy based on the maximum a posteriori probability (MAP) estimate for discrete weighted bipartite graphs that can encompass the BiSBM as a special case has recently been proposed by \cite{discrete-side-info}. However, their method requires estimating the latent partition of the columns and hence requires the sparsity level of the graph $p_{max}$ to satisfy $p_{max}\gtrsim \frac{\log n_2}{n_1}$. In contrast, we need a far weaker condition $p_{max}\gg \frac{1}{\sqrt{n_1n_2}}$ in the high dimensional regime. This highlights one of the difficulties we face in the high dimensional regime -- while it is impossible to correctly estimate all the model parameters, it is still possible to exactly recover the row partition. A similar phenomenon is known for Gaussian mixture models, see \cite{Ndaoud2018SharpOR}.

Our algorithm design is based on the Generalized Power Method (GPM) which has been applied successfully in various statistical learning problems in recent years. This includes, e.g., group synchronization \citep{boumal2016}, joint alignment from pairwise difference \citep{chen2016_alignment}, graph matching \citep{gpmmatching}, low rank matrix recovery \citep{chi2019} and SBM \citep{Wang2021OptimalNE}.

The work of \cite{Ndaoud2021ImprovedCA} which we extend in the present paper is also based on the GPM. However, there are significant differences between the algorithms. We do not need the centering step used in \citep{Ndaoud2021ImprovedCA} and since we encode the community memberships in a $n_1\times K$ matrix (instead of using the sign), our algorithm can be applied when $K>2$. Our algorithm is closer to the one proposed by \cite{Wang2021OptimalNE} for clustering graphs under the SBM. In contrast to \cite{Wang2021OptimalNE}, we do not add any constraints on the columns, and instead of solving a linear assignment problem, we directly project on the extreme points of the unit simplex. 

The GPM is also related to alternating optimization, a common strategy used to solve non-convex optimization problems in an iterative way. For example, EM-type algorithms \citep{Dempster1977MaximumLF} have been used since decades. In general, these methods are not guaranteed to achieve a global optimum. However, a recent line of research has show that under various statistical models, alternating optimization can actually lead to consistent estimators \citep{ Lu2016StatisticalAC, chen2018_retrieval, Gao2019IterativeAF,Han2020ExactCI, Chen2021OptimalCI, braun2021iterative}. Our proof techniques are based on the work of \cite{braun2021iterative} which itself is based on a general framework developed by \cite{Gao2019IterativeAF}.

\subsection{Notation}

We use lowercase letters ($\epsilon, a, b, \ldots$) to denote scalars and vectors, except for universal constants that will be denoted by $c_1, c_2, \ldots$ for lower bounds, and $C_1, C_2, \ldots $ for upper bounds and some random variables. We will sometimes use the notation $a_n\lesssim b_n$ (or $a_n\gtrsim b_n$ ) for sequences $(a_n)_{n \geq 1}$ and $(b_n)_{n \geq 1}$ if there is a constant $C>0$ such that $a_n \leq C b_n$ (resp. $a_n \geq C b_n$) for all $n$. If the inequalities only hold for $n$ large enough, we will use the notation $a_n=O(b_n)$ (resp. $a_n=\Omega(b_n)$).  If $a_n \lesssim b_n$ (resp. $a_n=O(b_n)$) and $a_n \gtrsim b_n$ (resp. $a_n=\Omega(b_n)$), then we write $a_n \asymp b_n$ (resp. $a_n=\Theta(b_n)$).  

Matrices will be denoted by uppercase letters. The $i$-th row of a matrix $A$ will be denoted as $A_{i:}$. The column $j$ of $A$ will be denoted by $A_{:j}$, and the $(i,j)$th entry by $A_{ij}$. The transpose of $A$ is denoted by $A^\top$ and $A_{:j}^\top$ corresponds to the $j$th row of $A^\top$ by convention. $I_k$ denotes the $k\times k$ identity matrix. For matrices, we use $|| . ||$ and $||.||_F$ to respectively denote the spectral norm (or Euclidean norm in case of vectors) and Frobenius norm.  The number of non-zero entries of $A$ is denoted by $\texttt{nnz}(A)$. The vector of $\R^n$ with all entries equal to one is denoted by $\mathbf{1}_n$. When applied to a vector $x$ of length $K$, $\diag (x)$ will denote the diagonal $K\times K$ matrix formed with the entries of $x$.

\section{The statistical framework} \label{sec:stat_frame}

The BiSBM is defined by the following parameters.\begin{itemize}
    \item A set of nodes of type I, $\calN_1 =[n_1]$, and a set of nodes of type II, $\calN_2 =[n_2]$.
    
    \item A partition of $\calN_1$ into $K$ communities $\calC_1,\ldots, \calC_K$ %of sizes $n_1, \ldots, n_K$ 
    %of equal sizes $n_1/K$ 
    and a partition of $\calN_2$ into $L$ communities $\calC_1',\ldots, \calC'_L$. 
    %of sizes $n'_1, \ldots, n'_L $. 
    %of equal sizes $n_2/L$.
    
    \item Membership matrices $Z_1 \in \calM_{n_1,K}$ and $Z_2 \in \calM_{n_2,L}$ where $\calM_{n,K}$ denotes the class of membership matrices with $n$ nodes and $K$ communities.  Each membership matrix $Z_1\in \calM_{n_1,K}$ (resp. $Z_2\in \calM_{n_2,L}$) can be associated bijectively with a partition function $z:[n]\to [K]$ (resp. $z':[n]\to [L]$) such that $z(i)=z_i=k$  where $k$ is the unique column index satisfying $(Z_1)_{ik}=1$ (resp. $(Z_2)_{ik}=1$ ). To each matrix $Z_1 \in \mathcal{M}_{n_1,K}$ we can associate a matrix $W$ by normalizing the columns of $Z_1$ in the $\ell_1$ norm: $W=Z_1 D^{-1}$ where $D= Z_1^\top \mathbf{1}_{n_1}$.  This implies that $W^\top Z_1 =I_K= Z_1^\top W$. %\hemant{To be sure, we never defined $z_2$ globally, so we can think of the right notation later. It might be better to use $z'$ rather than $z_2$...}
    
    \item A connectivity matrix of probabilities between communities $$\Pi=(\pi_{k k'})_{k\in [K], k'\in [L]} \in [0,1]^{K \times L}.$$ 
\end{itemize}

%Let us denote $c(K,L)>0$
Let us write $P=(p_{ij})_{i,j \in [n] }:=Z_1\Pi (Z_2)^\top \in [0,1]^{n_1\times n_2}$. A graph $\calG$ is distributed according to BiSBM$(Z_1, Z_2, \Pi)$ if the entries of the corresponding bipartite adjacency matrix $A$ are generated by 
\[ 
A_{ij}  \overset{\text{ind.}}{\sim} \mathcal{B}(p_{ij}), \quad i \in [n_1], \ j \in [n_2],
\] 
where $\calB(p)$ denotes a Bernoulli distribution with parameter $p$. Hence the probability that two nodes are connected depends only on their community memberships. We will frequently use the notation $E$ for the centered noise matrix defined as $E_{ij}=A_{ij}-p_{ij}$, and denote the maximum entry of $P$ by $p_{max} = \max_{i,j} p_{ij}$. The latter can be interpreted as the sparsity level of the graph. We make the following assumptions on the model. 
\begin{assumption}[Approximately balanced communities]\label{ass:balanced_part}
The communities  $\calC_1,\ldots, \calC_K$, 
(resp. $\calC_1',\ldots, \calC'_L$) are approximately balanced, i.e., there exists a constant $\alpha \geq 1$ such that for all $k\in [K]$ and $l\in [L]$ we have 
\[ 
\frac{n_1}{\alpha K}\leq \abs{\calC_k} \leq \frac{\alpha n_1}{K} \text{ and } \frac{n_2}{\alpha L}\leq \abs{\calC'_l} \leq \frac{\alpha n_2}{L}.
\] 
%\hemant{I don't think there was any need to give symbols to the size of the communities?}
\end{assumption}

\begin{assumption}[Full rank connectivity matrix]\label{ass:eig_low_bd}
The smallest eigenvalue of $\Pi\Pi^\top$, denoted by $\lambda_K(\Pi\Pi^\top)$, satisfies $\lambda_K(\Pi\Pi^\top) \gtrsim p_{\max}^2$ .
% and \rev{$\norm{\Pi}\leq c(K,L) p_{max}$ where $c(K,L)$ is a function of $K$ and $L$}. 
%\hemant{The upper bound is not assumption, we will handle in analysis.}
\end{assumption}
%
%
%\hemant{I think the upper bound on norm of $\Pi$ in the above assumption is not really an assumption. So it is better to state it as an equation outside the assumption environment. Then in the statements of the theorems/lemmas later, we can refer to that equation.}
%\gb{I added a reference to Remark 1 when needed.}

\rev{
\begin{remark}\label{rmk:1}
The following remarks are in order.
\begin{itemize}
    \item In the analysis, we will need an upper bound on $\norm{\Pi}$. Note that $\norm{\Pi}\leq c(K,L) p_{max}$ is always satisfied for $c(K,L)=\sqrt{KL}$ since \[ \norm{\Pi}\leq \norm{\Pi}_F\leq \sqrt{KL}p_{max}.\] But in some cases we might have a tighter bound, for e.g., if $\Pi$ is an (approximately) diagonal matrix.
    \item The spectral gap assumption $\lambda_K(\Pi\Pi^\top) \gtrsim p_{\max}^2$ is common in the spectral clustering literature (see for example \cite{lei2015}). However, it should be possible to recover the $K$ communities even if $\Pi$ is rank deficient (for example if one row of $\Pi$ is proportional to another row of $\Pi$ and the gap between these rows is large enough). Unfortunately, except for the isotropic Gaussian mixture setting \citep{loffler21}, it is not known how to analyze the spectral method without the spectral gap assumption.
    \item Instead of using the spectral gap assumption to quantify the signal, we could have used the minimal CH-divergence between the rows of $\Pi \Pi^\top$, similarly to the work of \cite{abbesadon15}. This approach is considerably more technical and would lead to a similar result under the Symmetric BiSBM (defined below).
\end{itemize}
\end{remark}
}

\begin{assumption}[Diagonal dominance]\label{ass:diag_dom} 
There exist $\beta > 0$ and $\eta \geq 1$ such that for all $k' \neq k \in [K]$, we have
%\[ \sum_{l\in [L]} |\calC_{l}'|\left(\Pi_{k'l}^2  - \Pi_{k'l}\Pi_{kl} \right)\geq \frac{\beta n_2 p_{max}^2}{\alpha L} \] 
\[(\Pi\Pi^\top)_{kk}- \alpha^2(\Pi\Pi^\top)_{kk'}\geq \beta p_{max}^2 \quad \text{ and } \quad (\Pi\Pi^\top)_{kk}- \frac{1}{\alpha^2}(\Pi\Pi^\top)_{kk'} \leq \eta \beta p_{max}^2.\]
\end{assumption}

%\hemant{Should use $k,k'$ for $[K]$; $l,l'$ for $[L]$; $i,j$ for $[n_1], [n_2]$, etc.}

\begin{remark}
When all communities $\calC'_{l}$ have size equal to $n_2/L$, the first condition in Assumption \ref{ass:diag_dom} simplifies to $(\Pi\Pi^\top)_{kk}- (\Pi\Pi^\top)_{kk'}\geq \beta p_{max}^2$ and corresponds to a diagonal dominance assumption. 
%When the communities are unbalanced, one can easily check that if $(\Pi\Pi^\top)_{k'k'}- \alpha^2(\Pi\Pi^\top)_{k'k}\geq \beta p_{max}^2$ then the first condition on Assumption \ref{ass:diag_dom} is also satisfied. 
The second condition in Assumption \ref{ass:diag_dom} is useful to show that $\beta$ is (up to a parameter $\eta$ that could depend on $L$ as discussed in the remark below) the parameter that measures the minimum difference between the diagonal and off-diagonal entries of $\Pi\Pi^\top$.
%For simplicity we introduce only one parameter $\alpha$ to measure how far the communities are from being balanced. Our algorithm relies in an essential way on the fact that $\Pi \Pi^\top$ is diagonally dominant. The parameter $\beta$ measure the minimal separation between diagonal and off-diagonal coefficients. 
\end{remark}

%\hemant{interpret the assumptions in words. Also, the condition $n_2 \gg n_1 \log n_1$ should be inline}

\paragraph{Symmetric BiSBM.} A particular case of interest is  where $L=K$ and $\Pi = (p-q)I_K +q\indic_{K}\indic_{K}^\top$ where $1\geq p>q\geq 0$ and $q=c p$ for some constant $0 < c < 1$. This model will be referred to as the Symmetric BiSBM, denoted by SBiSBM$(Z_1,Z_2,p,q)$. For $K=2$ this corresponds to the model analyzed in the work of \cite{Ndaoud2021ImprovedCA}. Since we now have (for $k'\neq k$) 
$$(\Pi\Pi^\top)_{kk} = p^2+(K-1)q^2 \text{ and } (\Pi\Pi^\top)_{kk'} = 2qp + (K-2)q^2$$  the following observations are useful to note. 
\begin{itemize}
\item We have $\norm{\Pi}\leq Kp_{max}$, so $c(K,L)=K$.
\item  We have $(\Pi\Pi^\top)_{kk}-(\Pi\Pi^\top)_{kk'} = (p-q)^2 \leq p_{max}^2$, and also $(p-q)^2 \geq (1-c)^2 p_{max}^2$. Hence Assumption \ref{ass:diag_dom} is satisfied for $\eta = \frac{1}{(1-c)^2}$ and $\beta = (1-c)^2$ in the equal-size community case ($\alpha=1$). Additionally, $\lambda_K(\Pi\Pi^\top) = (1-c)p_{\max}$.
%For $\alpha \neq 1$ we have $(\Pi\Pi^\top)_{kk}-\alpha^2(\Pi\Pi^\top)_{kk'}\geq \alpha^2 ((\Pi\Pi^\top)_{kk}-(\Pi\Pi^\top)_{kk'})$ and we can use the previous bound.
%\hemant{I fixed a typo in definition of $\Pi$. We should discuss how the above assumptions are satisfied in the special case.}
%\gb{I've detailed the computation in Section 4.2.1, but for completeness I added the result there also.}

\item In the unequal-size case, we need to show choices of $\beta, \eta$ under which Assumption \ref{ass:diag_dom} holds. To this end, note that
\begin{equation*}
  (\Pi\Pi^\top)_{kk}-\alpha^2(\Pi\Pi^\top)_{kk'} = p_{\max}^2 ((1-c)^2 - (\alpha^2-1)K),  
\end{equation*}
and 
\begin{equation*}
  (\Pi\Pi^\top)_{kk}-\frac{1}{\alpha^2}(\Pi\Pi^\top)_{kk'} = p_{\max}^2 ((1-c)^2 - (\frac{1}{\alpha^2}-1)K)\leq p_{max}^2((1-c)^2 + (\alpha^2-1)K).  
\end{equation*}
Hence if, for instance, $\alpha^2 = 1+\frac{(1-c)^2}{2K}$, we see that Assumption \ref{ass:diag_dom} is satisfied for $\beta = \frac{(1-c)^2}{2}$ and $\eta = 3$.
%
%\hemant{there is something weird here, needs to be written properly to reflect the requirement on $\beta$.}
\end{itemize}

\begin{remark}
One of the parameters $\beta$ or $\eta$ can scale with $L$ as shown in the following examples. Assume $L$ is even and $\alpha=1$. Let us consider $\mathbf{p}=(p_1, \ldots ,p_{L/2})$, $\mathbf{q}=(q_1, \ldots ,q_{L/2})$ where $p_l>q_l$ for all $l\leq L/2$ and
\[\Pi=
\begin{pmatrix}
\mathbf{p}& \mathbf{q} \\
\mathbf{q}& \mathbf{p} 
\end{pmatrix}\in [0,1]^{2\times L}.
\]
 Then \[ L\max_l(p_l-q_l)^2\geq \min_{k\neq k'}(\Pi\Pi^\top)_{kk}-(\Pi\Pi^\top)_{kk'}=\norm{p-q}^2\geq L\min_l(p_l-q_l)^2,\]  hence $\beta = \Theta (L)$ and $\eta=\Theta(1)$. Let us define $\mathbf{q'}=(q_1',q_2',q_3, \ldots, q_{L/2})$ for $q_1'=0.5q_1$ and $q_2'=0.5q_2$ If we now consider \[\Pi=
\begin{pmatrix}
\mathbf{p}& \mathbf{q} \\
\mathbf{q}& \mathbf{p} \\
\mathbf{q}'& \mathbf{p} \\
\end{pmatrix} \in [0,1]^{3\times L}.
\]
where $q_i$ and $p_i$ are as before and $q'_1\neq q_1$, $q'_2\neq q_2$. Then it is easy to check that $\max_{k\neq k'}(\Pi\Pi^\top)_{kk}-(\Pi\Pi^\top)_{kk'} \leq Lp_{max}^2$ but \[ \min_{k\neq k'}(\Pi\Pi^\top)_{kk}-(\Pi\Pi^\top)_{kk'} = (\Pi\Pi^\top)_{22}-(\Pi\Pi^\top)_{23}= \frac{q_1^2+q_2^2}{2}\] is independent of $L$. Hence $\beta=\Theta(1)$ and $\eta=\Theta(L)$.
\end{remark}

\paragraph{Error criterion.} The \textbf{misclustering rate} associated to an estimated partition $\hat{z}$ is defined by \begin{equation}
 \label{eq:def_misclust}   
 r(\hat{z},z)=\frac{1}{n}\min _{\pi \in \mathfrak{S}}\sum_{i\in [n]} \indic_{\lbrace \hat{z}(i)\neq \pi(z(i))\rbrace},\end{equation} where $\mathfrak{S}$ denotes the set of permutations on $[K]$. It corresponds to the proportion of wrongly assigned nodes labels.
%It is linked to the Hamming distance \[ h(z,z') = \sum_{i\in [n]} \indic_{\lbrace z(i) \neq z'(i) \rbrace}.\] 
We say that we are in the \textbf{exact recovery} regime if $r(\hat{Z},Z)=0$ with probability $1-o(1)$ as $n$ tends to infinity. If $\prob(r(\hat{Z},Z)=o(1))=1-o(1)$ as $n$ tends to infinity then we are in the \textbf{weak consistency} or \textbf{almost full recovery regime}. A more complete overview of the different types of consistency and the sparsity regimes where they occur can be found in \cite{AbbeSBM}.

\begin{remark}
As already mentioned in the introduction, we will focus in this work on the regime where $n_2\gg n_1 \log n_1$ and $\sqrt{n_1n_2}p_{max}\gtrsim \sqrt{\log n_1}$. In this parameter regime there is no hope to accurately recover $Z_2$ because the columns of $A$ are too sparse. Indeed, consider the setting where $\sqrt{n_1n_2}p_{max} \asymp \sqrt{\log n_1}$. Then in expectation, the sum of the entries of each column is $n_1p_{max} \asymp \sqrt{\frac{n_1\log n_1}{n_2}} \to 0$. But by analogy to the SBM, we would need a condition $n_1p_{max}\to \infty$ in order to recover $Z_2$. This is actually a necessary condition obtained in a related setting by \cite{discrete-side-info}. While it is impossible to estimate $Z_2$ in this sparsity regime, it is still possible to accurately estimate $Z_1$. We will focus on this problem from now onwards.
%goes to $\infty$ and $\sqrt{n_1n_2}p_{max}=O(\log n_1)$. \hemant{You mean $\sqrt{n_1n_2}p_{max} \gtrsim \log n_1$? But then there seems to be some redundancy in also saying that  $\sqrt{n_1n_2}p_{max} \rightarrow \infty$? Also, we should use the big O only for upper bound. The symbol $\Omega$ is reserved for lower bound and $\Theta$ for lower and upper bound.}
\end{remark}

%\hemant{We need to mention somewhere that our goal is to estimate $Z_1$ using $A$, and recall that in the above parameter regimes, there is no hope to recover $Z_2$.}
\section{Algorithm}
\subsection{Initialization with a spectral method}
We can use a spectral method on the hollowed Gram matrix $B$ to obtain a first estimate of the partition $Z_1$. This is similar to the algorithm in \cite{florescu16}; we only use a different rounding step so the algorithm can be applied to general bipartite graphs with $K>2$ communities, in contrast to \cite{florescu16}.

\begin{algorithm}[hbt!]
\caption{Spectral method on $\calH(AA^\top)$ (\spec)}\label{alg:spec}
\begin{flushleft}
        \textbf{Input:} The number of communities $K$ and the adjacency matrix $A$.
\end{flushleft}
        \begin{algorithmic}[1]
        \State Form the diagonal hollowed Gram matrix $B:=\calH(AA^\top) $ where $\calH(X)=X-\diag(X)$.
       \State Compute the matrix $U\in \R^{n_1\times K}$ whose columns correspond to the top $K$-eigenvectors of $B$.
       \State Apply approximate $(1+\bar{\epsilon})$ approximate $\texttt{k-means}$ on the rows of $U$ and obtain a partition $z^{(0)}$ of $[n_1]$ into $K$ communities.
       \end{algorithmic}
 \textbf{Output:} A partition of the nodes $z^{(0)}$.
\end{algorithm}

\paragraph{Computational complexity of \spec.} The cost for computing $B$ is $O(n_1\texttt{nnz}(A))$  
%Since $A$ is sparse, $\texttt{nnz}(A)=O(n_1n_2p_{max})\ll n_1n_2$. 
\renewcommand{\thefootnote}{\arabic{footnote}}
and for $U$ is\footnote{The $\log n_1$ term comes from the number of iterations needed when using the power method to compute the largest (or smallest) eigenvector of a given matrix. 
%\hemant{check the statement; note that other eigenvectors can be found by deflation process...}
} $O(n_1^2K\log n_1)$.
% \hemant{Isn't it just $n_1^2 K$?}
 Applying the $(1+\bar{\epsilon})$ approximate $\texttt{k-means}$ has a complexity $O(2^{(K/\bar{\epsilon})^{O(1)}}n_1K)$, see \cite{approxkmeans}. In practice, this operation is fast and the most costly operation is the computation of $B$ which has complexity $O(n_1^2n_2p_{max})$, since one can show that $\texttt{nnz}(A)=O(n_1n_2p_{max})$ with high probability.

\subsection{Iterative refinement with GPM}
Our algorithm is based on the Generalized Power Method. In contrast to the power method proposed recently by \cite{Ndaoud2021ImprovedCA}, we do not require to center the adjacency matrix $A$ and, instead of using the sign to identify the communities, we project on $\calM_{n_1,K}$.  Consequently, our algorithm can be applied to bipartite graphs with $K>2$ and $K\neq L$ while \cite{Ndaoud2021ImprovedCA} require $K = L = 2$.

In the first step we form the diagonal hollowed Gram matrix $B$. This is natural since in the parameter regimes we are interested in, there is no hope to consistently estimate $Z_2$. Then, we iteratively update the current estimate of the partition $Z_1^{(t)}$ by a row-wise projection of $B Z_1^{(t)}$ onto the extreme points of the unit simplex of $\R^K$ denoted by $\calS_K$. The  projection operator $\calP: \matR^K \rightarrow \calS_K$ is formally defined by 
\begin{equation}\label{def:operator}
    \calP (x):= \arg\min_{y\in \calS_K} \norm{x-y} \text{ for all }x\in \R^K.
\end{equation}
This implies that $(Z_1^{(t+1)})_{ik}=1$ for the value of $k$ that maximizes $B_{i:}Z_{:k}^{(t)}$. Ties are broken arbitrarily.  

\begin{algorithm}[hbt!]
\caption{Generalized Power Method (\gpm)}\label{alg:gpm}
\begin{flushleft}
        \textbf{Input:} Number of communities $K$, adjacency matrix $A$, an estimate of the row partition $\hat{Z}^{(0)}_1$ and,  number of iterations $T$.
\end{flushleft}
        \begin{algorithmic}[1]
        \State Form the diagonal hollowed Gram matrix $B:=\calH(AA^\top) $ where $\calH(X)=X-\diag(X)$ and compute $W^{(0)}=(\hat{Z}_1^{(0)})^\top (D^{(0)} )^{-1}$ where $D^{(0)}=\diag( (\hat{Z}_1^{(0)})^\top\mathbf{1}_{n_1})$.
        \For{$0\leq t \leq T$}
            \State Update the partition $Z_1^{(t+1)}=\calP (BW^{(t)})$ where $\calP$ is the operator in \eqref{def:operator} applied row-wise.
            
            \State Compute $D^{(t+1)}=\diag ((Z_1^{(t+1)})^\top\mathbf{1}_{n_1})$.
            Then form $W^{(t+1)}=(Z_1^{(t+1)})^\top (D^{(t+1)} )^{-1}.$ 
        \EndFor
        \end{algorithmic}
 \textbf{Output:} A partition of the nodes $Z_1^{(T+1)}$.
\end{algorithm}

\paragraph{Computational complexity of \gpm.} As mentioned earlier, the cost for computing $B$ is $O(n_1\texttt{nnz}(A))$. For each $t$, the cost of computing $BW^{(t)}$ is $O(n_1^2K)$ and the cost of the projection is $O(K n_1)$.  The cost of computing $D^{(t+1)}$ and $W^{(t+1)}$ is $O(Kn_1)$.
%\hemant{Mention cost of update step $4$ as well.} 
So the total cost over $T$ iterations is $O(Tn_1^2K)$ and doesn't depend on $n_2$. 

\begin{remark}The recent work of \cite{Wang2021OptimalNE} proposed a power method for clustering under the SBM, but instead of using a projection on the simplex as we did, they add an additional constraint on the column on $Z$ so that each cluster has the same size. They showed that computing this projection is equivalent to solving a linear assignment problem (LAP). But the cost of solving this LAP is $O(K^2n_1\log n_1)$ whereas the cost of the projection on the simplex is $O(Kn_1)$. Moreover, their algorithm requires to know in advance the size of each cluster and it is not straightforward to extend their theoretical guarantees to the approximately balanced community setting.
%Notice that the cost of the projection is less than the cost of solving a linear assignment problem as in \cite{Wang2021OptimalNE}, which is $O(K^2n_1\log n_1)$. 
%\hemant{Elaborate more on this, since the reader will not know what Wang et al. are doing. }
\end{remark}

\section{Spectral initialization}
We show that Algorithm \ref{alg:spec} can recover an arbitrary large proportion of community provided that the sparsity level $p_{\max}$ is sufficiently large, and $n_2/n_1 \rightarrow \infty$ sufficiently fast, as $n_1 \rightarrow \infty$.

%\hemant{need to clarify the word ``detection'', maybe some rewording needed.}

\begin{proposition}\label{prop:spec}
Assume that $A \sim BiSBM(Z_1,Z_2,\Pi)$ and Assumptions \ref{ass:balanced_part},  \ref{ass:eig_low_bd} are satisfied. For any $0 < \varepsilon < 1$, suppose additionally that 
\begin{align} 
    \frac{n_2}{n_1} 
    &\gtrsim \max \left(L \log n_1, \rev{\frac{\log^2 n_2}{\log n_1}} \right), \quad p_{\max}^2 \geq \rev{\alpha^{2.5}} \frac{KL \log n_1}{\varepsilon n_1n_2}, \label{eq:n12_pmax_conds_spec} \\
    \rev{\alpha^3}\frac{KL}{\log n_1} &= o(1), \quad \rev{\alpha^{12}}\frac{(KL)^3}{\frac{n_2}{n_1} \log n_1} = o(1) \quad \text{ as } n_1 \rightarrow \infty. \label{eq:n12_KL_conds_spec}
\end{align}
%We further assume that $p_{max}^2\geq C(K,L) \frac{\log n_1}{n_1n_2}$ where $C(K,L) \geq K^2L \epsilon'$ for a small enough constant $\epsilon '>0$.  
Then the output $z^{(0)}$ of Algorithm \ref{alg:spec} satisfies  with probability at least $1 - n^{-\Omega(1)}$ the bound 
$r(z^{(0)},z) \lesssim K\varepsilon^2.$
\end{proposition}

\begin{proof}
First, we will control the noise $\norm{B-PP^\top}$ using Lemma \ref{lem:conc_mat}. Note that the conditions in \eqref{eq:n12_pmax_conds_spec} imply the conditions on $n_1,n_2$ and $p_{\max}$ in Lemma \ref{lem:conc_mat}. Since $\calH(\cdot)$ is a linear operator and $B = \calH(B)$, we have the decomposition 
\begin{align*}
    B = \calH(PP^\top) + \calH(PE^\top) + \calH(EP^\top) + \calH(EE^\top),
\end{align*}
which leads to the following bound by triangle inequality
\begin{align}
    \norm{B - PP^\top} 
    &\leq \norm{\calH(PE^\top) + \calH(EP^\top) } + \norm{\calH(EE^\top)} + \norm{\calH(PP^\top) - PP^\top} \nonumber \\
    &\leq 4\norm{E Z_2}\norm{\Pi^\top Z_1^\top}+ \norm{\calH(EE^\top)} + \norm{\calH(PP^\top) - PP^\top}. \label{eq:B_pp_bd1}
\end{align}
Now let us observe that 
\begin{align*}
 \norm{\calH(PP^\top) - PP^\top} &= \norm{\calH(Z_1\Pi Z_2^\top Z_2\Pi^\top Z_1^\top) - Z_1 \Pi Z_2^\top Z_2\Pi^\top Z_1^\top} \\
 &=  \norm{\diag(Z_1 \Pi Z_2^\top Z_2\Pi^\top Z_1^\top)} \\
 &\leq \frac{\alpha n_2}{L} (L p_{\max}^2) \\
 &\leq  \alpha n_2 p_{\max}^2, 
\end{align*}
while,
\begin{equation*}
  \norm{\Pi^\top Z_1^\top} \leq \norm{Z_1} \norm{\Pi} \leq \rev{\sqrt{\frac{\alpha n_1}{K}} \norm{\Pi}_F \leq \sqrt{\alpha n_1 L} p_{\max}}. \tag{\rev{by the upper bound in Remark \ref{rmk:1}}}
\end{equation*}
Using these bounds along with the bounds on $\norm{E Z_2}$ and $\norm{\calH(EE^\top)}$ from Lemma \ref{lem:conc_mat}, and applying them in \eqref{eq:B_pp_bd1}, it holds with probability at least $1-n^{-\Omega(1)}$ that
\begin{align*}
    \norm{B - PP^\top} \lesssim \max(\log n_1,  \sqrt{n_1 n_2} p_{\max}) + \rev{\alpha}\sqrt{n_2} n_1 p_{\max}^{1.5} + \rev{\alpha} n_2 p_{\max}^2.
\end{align*}
%
%
%By Lemma \ref{lem:conc_mat} we have $\norm{EE^\top -\expec(EE^\top) } \lesssim \log n_1$ w.h.p. It is easy to see that the matrix $EZ_2^\top \in \R^{n_1\times L}$ has independent entries of the form $\sum_{j\in \calC_l'}E_{ij}$ with variance $\sigma_{il}^2$ at most $p_{max}n_2/L$. By using the concentration result in Lemma \ref{lem:conc_mat}, we get $||EZ_2|| \lesssim \sqrt{n_1n_2p_{max}/L}$.  Consequently,  
%
%\begin{align*} 
%\norm{B-PP^\top} &\leq 2\norm{EE^\top -\expec(EE^\top)} + 4\norm{EZ_2}\norm{\Pi Z_1^\top} + \norm{\calH(\Pi \Pi^\top)-\Pi \Pi^\top} \\
%&\lesssim \log n_1 + \sqrt{n_2}n_1p_{max}^{1.5}/L +n_2p_{max}^2/L \\
%&\lesssim \log n_1.
%\end{align*} 
%
Let us denote by $\hat{U}$ (resp. $U$) to be the matrix of top-$K$ eigenvectors of $B$ (resp. $PP^\top$). Denoting $\lambda_K(PP^\top)$ to be the $K$th largest eigenvalue of $PP^\top$, it is not difficult to verify that 
$$\rev{\lambda_K(PP^\top) \gtrsim  \lambda_L(Z_2)^2 \lambda_K(Z_1)^2 \lambda_K(\Pi)^2 \gtrsim \frac{n_1 n_2p_{max}^2}{\alpha^2 KL}.}$$ 
Then by the Davis-Kahan Theorem \citep{Yu2014AUV}, there exists an orthogonal matrix $Q\in \R^{K\times K}$ such that
\begin{align} 
\norm{\hat{U}-UQ}\lesssim \frac{\norm{B-PP^\top}}{\lambda_K(PP^\top)} \lesssim \rev{\alpha^2}\frac{KL \log n_1}{n_1n_2 p_{\max}^2} + \rev{\alpha^2} \frac{KL}{\sqrt{n_1n_2}p_{\max}} + \rev{\alpha^3}\frac{KL}{\sqrt{n_2p_{\max}}} + \rev{\alpha^3} \frac{KL}{n_1}, \label{eq:dk_thm_bd}
%\frac{L\log n_1}{n_1n_2p_{max}^2} = \frac{LK}{C(K,L)} 
\end{align}
where we bounded the $\max$ operation by the sum. Now on account of the conditions in \eqref{eq:n12_pmax_conds_spec} and  \eqref{eq:n12_KL_conds_spec}, it is easily seen that $\rev{\alpha^2}\frac{KL \log n_1}{n_1n_2 p_{\max}^2} \lesssim \varepsilon$, while the remaining three terms in the RHS of \eqref{eq:dk_thm_bd} are $o(1)$ as $n_1 \rightarrow \infty$.
%
%because  
%\gb{Here we used the fact that $\alpha$ and $\beta$ are constants. But it is not obvious how $\lambda_K(\Pi\Pi^\top)$ will depends on these quantities.}
Finally, we conclude by using the same argument as in the proof of Theorem $3.1$ in \cite{lei2015} to show that \[ 
r(\hat{z},z) \lesssim \norm{\hat{U}-UQ}_F^2 \lesssim K \norm{\hat{U}-UQ}^2 \lesssim  K\varepsilon^2.
\]
%\hemant{Need to check this again from Lei and Rinaldo although I think it is correct. But not sure if Theorem $1$ there is the right reference? Also, the bound is $K \varepsilon^2$, not $\varepsilon$ as written previously, so need to check if it affects the ensuing GPM analysis.}
\end{proof}

%\begin{remark}The additional constraint $n_2p_{max}\gtrsim \log n_1$ is mild. It is satisfied for example if $n_2\gtrsim n_1 \log ^2n_1$ and $p_{max}\gtrsim \frac{1}{\sqrt{n_1n_2}}$ or if $p_{max}\geq \sqrt{\frac{\log n_1}{n_1n_2}}$. Our result also improves the result obtained by \cite{Ndaoud2021ImprovedCA} for weak consistency. Indeed they require $n_1n_2p^2_{max}\gtrsim \log n_1 $ while we only need $n_1n_2p^2_{max} \to \infty$ (if $K, L$ are independent of $n$).\end{remark}

\begin{remark}
As shown in Section \ref{sec:xp}, \spec\, has very good empirical performance. This suggests that the previous proposition is far from being optimal and doesn't capture the true rate of convergence of this spectral method. Also note that Proposition \ref{prop:spec} gives a meaningful bound only when $\varepsilon =O(\frac{1}{\sqrt{K}})$.  
\end{remark}

\section{Analysis of \gpm}
%\hemant{Resume from here...}
Our analysis strategy is similar to the one recently considered by \cite{braun2021iterative} for clustering under the contextual SBM, which in turn  is based on the framework recently developed by \cite{Gao2019IterativeAF}. There are however some additional technical difficulties due to the fact that there are more dependencies in the noise since the matrix $B$ is a Gram matrix.

We will assume w.l.o.g. that $\pi^*$ the permutation that best aligns $z^{(0)}$ with see (see equation \eqref{eq:def_misclust}) is the identity, if not, then replace $z$ by $(\pi^*)^{-1}(z)$. Hence there is no label switching ambiguity in the community labels of $z^{(t)}$ because they are determined from $z^{(0)}$. 

First we will decompose the event  ``after one refinement step, the node $i$ will be incorrectly clustered given the current estimation of the partition $z^{(t)}$ at time $t$" into an event independent of $t$, and events that depend on how close $z^{(t)}$ is from $z$. Then we will analyze these events separately. Finally, we will use these results to show that the error contracts at each step.

\subsection{Error decomposition}
By definition, a node $i$ is misclustered at step $t+1$ if there exists a $k \neq z_i \in [K]$ such that \begin{equation}\label{eq:mis_cond}
     B_{i:}W_{:k}^{(t)}\geq B_{i:}W_{:z_i}^{(t)}.
\end{equation}
By decomposing $B$ as \[ B =\underbrace{\calH( PP^\top)}_{\Tilde{P}}+\underbrace{\calH(EP^\top+PE^\top +EE^\top)}_{\Tilde{E}}\] one can show that condition \eqref{eq:mis_cond} is equivalent to 
\begin{equation}\label{eq:mis_cond2}
    \Tilde{E}_{i:}(W_{:z_i}-W_{:k})\leq -\Delta^2(z_i,k)+F_{ik}^{(t)}+G_{ik}^{(t)}
\end{equation}  
where 
\begin{align*}
    \Delta^2(z_i,k) & = \Tilde{P}_{i:}(W_{:z_i}-W_{:k}),\\
    F_{ik}^{(t)}&=\langle\Tilde{E}_{i:}(W^{(t)}-W),e_{z_i}-e_{k} \rangle, \\
    \text{and }G_{ik}^{(t)}&= \langle\Tilde{P}_{i:}(W^{(t)}-W),e_{z_i}-e_{k} \rangle.
\end{align*}
Here  $e_1, \ldots, e_K$ denotes the canonical basis of $\R^{K}$. The terms $F_{ik}^{(t)}$ and $G_{ik}^{(t)}$ can be interpreted as error terms (due to $W^{(t)}\neq W$) while $\Delta^2(z_i,k)$ corresponds to the signal. Indeed, let us denote $Q=\Pi (Z_2^\top Z_2) \Pi^\top$. Then we obtain $PP^\top = Z_1 Q Z_1^\top$ which implies
\begin{align*}
    \Delta^2(z_i,k) =  \Tilde{P}_{i:}(W_{:z_i}-W_{:k})&= Q_{z_iz_i}-Q_{z_ik}-(PP^\top)_{ii}(W_{iz_i}-W_{ik})\\
     &=(1-\frac{K}{n_1})Q_{z_iz_i}-Q_{z_ik}. 
\end{align*}
%

%with equal size communities ($\alpha=1$), 
We have for all $k, k'\in [K]$ that \[\frac{n_2}{\alpha L}(\Pi\Pi^\top)_{kk'}\leq  Q_{kk'}\leq  \frac{\alpha n_2}{L}(\Pi\Pi^\top)_{kk'}\]
 which implies  \[\frac{n_2}{ \alpha L}\left((\Pi\Pi^\top)_{z_iz_i}-\alpha^2(\Pi\Pi^\top)_{z_ik}\right) \leq \Delta^2(z_i,k) \leq \frac{\alpha n_2}{  L}\left((\Pi\Pi^\top)_{z_iz_i}-\frac{1}{\alpha^2}(\Pi\Pi^\top)_{z_ik}\right) \]  for $n_1$ large enough. By Assumption \ref{ass:diag_dom} this implies  
\begin{equation}\label{eq:delta_zik}
    \beta \frac{n_2p_{max}^2}{\alpha L} \leq \Delta^2(z_i,k) \leq  \eta\beta \frac{\alpha n_2p_{max}^2}{L}.
\end{equation} 
when $n_1$ is large enough.

\subsection{Oracle error}
We want to show that the condition \eqref{eq:mis_cond2} cannot occur with high probability. First we will show that this is indeed the case when we ignore the $F$ and $G$ error terms; the subsequent error will be referred to as the oracle error 
\[ \xi(\delta) = \sum_{i=1}^{n_1}\sum_{k\in [K]\setminus z_i}\Delta^2(z_i,k)\indic_{\ov } \text{ for } \delta \in (0,1)
\] 
where
\[\ov = \left\lbrace  \tilde{E}_{i:}(W_{:z_i}-W_{:k}) \leq -(1-\delta)\Delta^2(z_i,k) \right\rbrace .\]
Let us denote 
\[\Delta_{min}=\min_{a\neq b \in [K]}\Delta(a,b)\] 
%
%\gb{
to be the minimal separation of the parameters associated with the different communities.  We will also denote by
\[ \tilde{\Delta}^2 := \frac{ \beta^2}{12eL\alpha ^3}\frac{n_1n_2p_{max}^2}{KL}, \]
%\[ \tilde{\Delta}^2 := \frac{n_1n_2(p-q)^2}{K^2}\]
%
%where $\epsilon =\beta/(eL\alpha)$ (see Lemma \ref{lem:oracle})
to be the approximate signal-to-noise ratio (SNR) associated with the model.

%For $t \geq 1$ and $\delta \in [0,1)$, let 
%\[ \delta^{(t)}= \max \left(\frac{7}{8}\frac{\tau^{(t-1)}}{\tau^{(0)}}, \delta \right), \ \tau^{(t)} = \tau^{(0)} \delta^{(t)} \] 
%
%
In general the rate of decay of $\xi(\delta)$ leads to the convergence rate of iterative refinement algorithms, hence it is important to control this quantity by showing that the following condition is satisfied. Let us define $$\tau^{(0)}=\epsilon' \delta \min(1, \beta^2) \frac{n_1\Delta_{min}^2}{K}$$  for a small enough constant $\epsilon'>0$. This parameter will be referred to as the minimal error required for the initial estimate of our algorithm. 
%\gb{Where should we define this quantity and explain its meaning ?}

\begin{condition}[oracle error]\label{cond:ideal} Assume that there exists $\delta\in (0,1)$ such that \[ \xi (\delta) \leq \frac{3}{4}\tau^{(0)} \] holds with probability at least $1-\eta_1$.
\end{condition}

\subsection{Contraction of the error at each step}
 Let 
 \[ l(z,z') = \sum_{i\in [n]} \Delta^2(z_i,z'_i)\indic_{\lbrace z_i\neq z_i'\rbrace } 
 \] 
 be a measure of the distance between two partitions $z, z'\in [K]^n$. We want to show that $\lt$ decreases until reaching the oracle error. To this end, we will need to control the noise level. In particular, we are going to show that the following two conditions are satisfied. 
 
\begin{condition}[F-error type]\label{cond:f}
Assume that \[ \max_{\lbrace z^{(t)}: l(z,z^{(t)})\leq \tau^{(0)} \rbrace}\sum_{i=1}^n\max_{b\in [K]\backslash z_i} \frac{(F_{ib}^{(t)})^2}{\Delta^2(z_i,b)l(z,z^{(t)})} \leq \frac{\delta^2}{256}\] holds with probability at least $1-\eta_2$.
\end{condition}

\begin{condition}[G-error type]\label{cond:g}
Assume that \[ \max_{i\in [n]}\max_{b \in [K]\setminus z_i} \frac{|G_{ib}^{(t)}|}{\Delta^2  (z_i,b)} \leq \frac{\delta}{4} \] holds uniformly on the event $\lbrace z^{(t)}: l(z,z^{(t)})\leq \tau^{(0)} \rbrace$  with probability at least $1-\eta_3$ .
\end{condition}

Under these conditions, we can show that the error contracts at each step. 

\begin{theorem}\label{thm:gao_ext}  Assume that $ l(z^{(0)},z) \leq \tau^{(0)}$.
Additionally assume that Conditions \ref{cond:ideal}, \ref{cond:f}, and \ref{cond:g} hold. Then with probability at least $1-\sum_{i=1}^3\eta_i$ 
\begin{equation}\label{eq:contraction}
    \lt \leq \xi(\delta) + \frac{1}{8}l(z^{(t-1)},z), \, \forall t \geq 1.
\end{equation} 
In particular, we have for all $t\gtrsim \log(1/\delta)$ that  \[ \lt \lesssim \xi(\delta)+ \tau^{(0)}(1/8)^{t-\Theta (\log(1/\delta))}. \] 
\end{theorem}

\begin{proof}
 It is an immediate adaptation of Theorem 3.1 in \cite{Gao2019IterativeAF}. The last part can be derived in the same way as in Corollary 1 in \cite{Braun2021ClusteringMG}.
\end{proof}

\subsection{Application to BiSBM}
When applied to the BiSBM the previous theorem leads to the following result. \rev{Recall $c(K,L)$ as defined within Remark \ref{rmk:1}.}
%\gb{The definition of $\tau^{(0}$ need to be changed appropriately. I think it should be $\epsilon \min(\beta^2,1)\frac{n_1\Delta_{min}^2}{K}$ with $\epsilon= \epsilon' \delta$}
\begin{theorem}\label{thm:sbisbm} Assume that  $A\sim BiSBM(Z_1,Z_2,\Pi)$,  $K^2L\leq \tilde{\Delta}^2$, $\tilde{\Delta}\to \infty$,  \eqref{eq:n12_pmax_conds_spec} and \eqref{eq:n12_KL_conds_spec} are satisfied. %\rev{Recall that $c(K,L)$ is such that $\norm{\Pi}\leq c(K,L)p_{max}$}. 
Under Assumptions \ref{ass:balanced_part}, \ref{ass:eig_low_bd} and \ref{ass:diag_dom}, if \gpm\, is initialized with a $z^{(0)}$ such that \[ \rev{ l(z,z^{(0)}) \leq \tau^{(0)}=\epsilon' \delta \min(1, \beta^2) \frac{n_1\Delta_{min}^2}{\alpha^{3}c(K,L)^2K}}\] for a small enough constant $\epsilon'>0$ and $\delta=\frac{1}{4\eta \alpha}$, then with probability at least $1-n_1^{-\Omega(1)}$ we have for all $t\gtrsim \log n_1$ \[ r(z^{(t)},z) \leq \exp (-(1-o(1))\tilde{\Delta}^2).\]  In particular, if $\tilde{\Delta}^2>\log n_1$, we can exactly recover $Z_1$.
\end{theorem}

\begin{corollary}
Under the assumptions of Theorem \ref{thm:sbisbm}, \spec\, returns an estimate $z^{(0)}$ such that \[\rev{l(z,z^{(0)}) \leq \eta \beta \frac{\alpha n_1n_2p_{max}^2}{L}r(z,z^{(0)})\lesssim  \eta \beta K^2\varepsilon^2 \frac{n_1 \Delta_{min}^2}{c(K,L)^2 K}}\] and hence satisfies $l(z,z^{(0)}) \leq \tau^{(0)}$ for $\rev{\varepsilon=O(\frac{ \min(1,\beta^2)}{Kc(K,L)\eta \sqrt{ \eta \beta}})}$.
\end{corollary}
%\hemant{Add corrolary for intersection of proposition and theorem}

\begin{remark} Specialized to the SBISBM case where $K=2=L$, $q=c p$ for some $0 < c < 1$, the condition $\tilde{\Delta}^2=\Omega(\log n_1)$ needed for exact recovery in the above theorem is equivalent to the condition $p_{max} = \Omega(\sqrt{\frac{\log n_1}{n_1n_2})}$ needed by \cite{Ndaoud2021ImprovedCA}. We will show in the next section that this rate is optimal. \end{remark}

\begin{remark}
The condition $n_1n_2p_{max}^2\gtrsim KL\log n_1$ is needed because we used the concentration inequality $||\calH(EE^\top)|| \lesssim   \max(\log n_1, \sqrt{n_1n_2}p_{max}) $ in Lemma \ref{lem:conc_mat}. If this last condition could have been replaced by $||\calH(EE^\top)|| \lesssim   \sqrt{n_1n_2}p_{max} $ then we would only need the condition \rev{$\tilde{\Delta}\geq c \log n_1$ for any constant $c>0$}. This is the principal step that needs to be improved in order to get more general weak consistency guarantees. \rev{A possible way to do that would be to extend the combinatorial techniques developed by \cite{freige05} for generalized Erdös-Renyi graphs to similarity matrices. When $n_1 n_2 p^2_{max} \to \infty$ at a rate slower than $\log n_1$, one could also try to extend the trimming method of \cite{vershynin_levina2015}.}
\end{remark}

% \begin{remark} In order to extend Theorem \ref{thm:sbisbm} to more general connectivity matrices $\Pi$ such that $\Pi \Pi^\top$ is diagonally dominant, the main difficulty is to generalize Lemma \ref{lem:oracle}. We believe that the same strategy can be used to bound the m.g.f., but the computations are more tedious. All the other steps remains the same.\end{remark}

\begin{proof}[Proof of Theorem \ref{thm:sbisbm}]
In order to apply Theorem \ref{thm:gao_ext}, we need to show that Conditions \ref{cond:ideal}, \ref{cond:f} and \ref{cond:g} are satisfied.

\paragraph{Oracle error.}
%Since for all $k,z_i\in [K]$, $\Delta^2(z_i,k) \approx p^2_{max}n_2/L $
Let $\delta= \frac{1}{4\eta \alpha}$. We obtain by Lemma \ref{lem:oracle} (see appendix) and Assumption \ref{ass:diag_dom} that \[ \expec(\xi(\delta)) \leq \alpha\eta\beta Kn_1 \frac{ n_2}{ L}p^2_{max}  e^{- \tilde{\Delta}^2} \lesssim \frac{\alpha^4}{\beta} K^2L \tilde{\Delta}^2e^{-\tilde{\Delta}^2} =e^{-(1-o(1))\tilde{\Delta}^2}\] by assumptions on $\tilde{\Delta}$ and $KL$. Then, by Markov inequality, we obtain \[ \prob \left(\xi(\delta)\geq \exp(\tilde{\Delta})\expec\xi(\delta)\right)  \leq \exp(-\tilde{\Delta}). \]
Consequently we get \[ \exp(\tilde{\Delta})\expec\xi(\delta) \leq \exp(-(1-o(1))\tilde{\Delta}^2)\] and hence with probability at least $1-\exp(-\tilde{\Delta})$ \[ \xi(\delta) \leq \exp(-(1-o(1))\Tilde{\Delta}^2)\leq \frac{3}{4}\tau^{(0)} \] because $\exp(-(1-o(1))\Tilde{\Delta}^2)=o(1) \ll \tau^{(0)}= \Omega(1)$ for $n_1$ large enough. This shows that Condition \ref{cond:ideal} is satisfied.
 
\paragraph{F-error term.}
With high probability, for all $z^{(t)}$ such that $\lt \leq \tau^{(0)}$ we have
\begin{align*}
    \sum_{i=1}^n\max_{b\in [K]\backslash z_i} \frac{(F_{ib}^{(t)})^2}{\Delta^2(z_i,b)l(z,z^{(t)})}& \leq \sum_i \norm{\tilde{E}_{i:}(W^{(t)}-W)}^2\max_{b\in [K]\backslash z_i}\frac{\norm{e_{z_i}-e_b}^2}{\Delta^2(z_i,b)l(z,z^{(t)})}\tag{by Cauchy-Schwartz}\\
    &\leq \frac{2\norm{\tilde{E}(W^{(t)}-W)}_F^2}{\Delta_{min}^2\lt}\\
    &\leq 2\frac{\norm{\tilde{E}}^2\norm{W^{(t)}-W}_F^2}{\Delta_{min}^2\lt}.
    \end{align*}
    
%    \hemant{$\norm{W^{(t)}-W}_F^2$ bounded by Lemma 4 and $\norm{\tilde{E}}^2$ bounded by Lemma 3. }
    
    By Lemma \ref{lem:conc_mat}, we have w.h.p. \[ \norm{\calH(EE^\top)} \lesssim \norm{EE^\top -\expec(EE^\top )} \lesssim \max(\log n_1,\sqrt{n_1n_2}p_{max})\] and \[ \norm{EP^\top}\leq \norm{EZ_2}\norm{\Pi Z_1^\top}\lesssim \rev{\alpha }\sqrt{\frac{n_1n_2p_{max}}{L}}\rev{c(K,L)}p_{max}\sqrt{\frac{n_1}{K}}=\rev{\alpha }(\sqrt{n_1n_2}\rev{c(K,L)}p_{max})\sqrt{\frac{n_1p_{max}}{KL}}.\]
    %But $n_1p_{max}=O(1)$ so the quadratic error  $\norm{\calH(EE^\top)}$ is dominating and we have $\norm{\tilde{E}}\lesssim \norm{\calH(EE^\top)}$.
    Since $\norm{\tilde{E}}\lesssim \norm{\calH(EE^\top)}+\norm{EP^\top}$ we obtain
    \begin{align*}
    \sum_{i=1}^n\max_{b\in [K]\backslash z_i} \frac{(F_{ib}^{(t)})^2}{\Delta^2(z_i,b)l(z,z^{(t)})} &\lesssim \frac{\max\left(\log ^2n_1,n_1n_2p_{max}^2, \rev{\alpha^2c(K,L)^2}(KL)^{-1}n_1^2n_2p_{max}^3\right) \rev{\alpha^3}K^3\lt }{n_1^3\Delta_{min}^6} %\tag{by Lemma \ref{lem:conc_mat}, Lemma \ref{lem:wt} }
    \\
    &\lesssim \rev{\alpha^3}K^2\frac{\max\left(\log ^2n_1,n_1n_2p_{max}^2, \rev{\alpha^5c(K,L)^2}(KL)^{-1}n_1^2n_2p_{max}^3\right)}{n_1^2\Delta_{min}^4}\frac{K\tau^{(0)}}{n_1\Delta_{min}^2}\\
    &\lesssim \frac{\rev{\alpha^5}K^2L^2\max\left(\log ^2n_1,n_1n_2p_{max}^2,\rev{\alpha^5c(K,L)^2}(KL)^{-1}n_1^2n_2p_{max}^3\right)}{\beta^2 n_1^2n_2^2p_{max}^4} \epsilon' \delta \beta^2 \tag{by definition of $\tau^{(0)}$ and \eqref{eq:delta_zik}}\\
    &\lesssim \frac{\rev{\alpha^5}K^2L^2\left(\log ^2n_1+n_1n_2p_{max}^2+\rev{\alpha^2 c(K,L)^2}(KL)^{-1}n_1^2n_2p_{max}^3\right)}{ n_1^2n_2^2p_{max}^4} \epsilon' \delta.
    %&\leq \frac{\delta}{256} \tag{by  assumption on the SNR and choice of $\epsilon'$. }
\end{align*}
\rev{From the assumptions of the theorem, one can verify that},
\[ 
     \frac{\rev{\alpha^5}K^2L^2\log ^2n_1}{ n_1^2n_2^2p_{max}^4} =O(1),\quad \frac{\rev{\alpha^7}K^2L^2}{n_1n_2p_{max}^2}=o(1) \quad  \text{ and } \quad \frac{\rev{\alpha^7} KL\rev{c(K,L)^2}}{n_2p_{max}}=o(1).\]
Hence for an appropriately small choice of $\epsilon'$ Condition \ref{cond:f} is satisfied.

\paragraph{G-error term.}
With high probability, for all $z^{(t)}$ such that $\lt \leq \tau^{(0)}$ we have
\begin{align*}
     \frac{|G_{i}^{(t)}|}{\Delta ^2(z_i,b) } &\leq \sqrt{2} \frac{\norm{\tilde{P}_{i:}(W^{(t)}-W)}}{\Delta_{min}^2}\tag{ by Cauchy-Schwartz }.
\end{align*}
But since $\norm{\tilde{P_{i:}}}\lesssim \norm{P_{i:}P^\top}\leq \norm{\Pi}^2 \norm{Z^\top_2Z_2}\norm{Z_1}\leq \rev{c(K,L)^2}\frac{\alpha n_2}{L}p_{max}^2\sqrt{\frac{\rev{\alpha}n_1}{K}}$
we obtain 
%\gb{Add a break between step 1 and 2}
\begin{align*}
     \frac{|G_{i}^{(t)}|}{\Delta^2(z_i,b) } &\lesssim \rev{\alpha^{1.5}c(K,L)^2}\frac{n_2p_{max}^2\sqrt{n_1}\norm{(W^{(t)}-W)}}{L\sqrt{K}\Delta_{min}^2}\\
     &\lesssim \rev{\alpha^{3}c(K,L)^2}\frac{n_2p_{max}^2}{L\Delta_{min}^2} \frac{K\lt}{n_1\Delta_{min}^2} \tag{ by Lemma \ref{lem:wt}}\\
     &\lesssim K\rev{\alpha^{3}c(K,L)^2} \frac{\lt}{\beta n_1 \Delta_{min}^2} \\
     &\lesssim \rev{\alpha^{3} c(K,L)^2}\frac{K \tau^{(0)}}{\beta n_1 \Delta^2_{\min}}. 
   % &\leq \frac{\delta}{4} \tag{by choice of $\epsilon'$ in the definition of $\tau^{(0)}$.}
\end{align*}

Now by choosing $\epsilon'$ to be a suitably small constant ($< 1$), we obtain

\begin{equation*}
    \frac{|G_{i}^{(t)}|}{\Delta^2(z_i,b)} \leq \frac{\delta}{4}.
\end{equation*}
 This shows that Condition \ref{cond:g} is satisfied.
\end{proof}

\section{Minimax lower bound}
%\gb{This is a work in progress section, you can skip it for the moment.}
Let us denote  the admissible parameters space \rev{for a SBiSBM (where $K = L$)} by
 \begin{align*}
     \Theta = \lbrace & P\in [0,1]^{n_1\times n_2}: P=Z_1\Pi Z_2^\top \text{ where }\Pi = q\indic_{K}\indic_{K}^\top +(p-q)I_K, 1>p>0,\\
     & q=cp \text{ for some constant }c\in (0,1), Z_1 \in \calM_{n_1,K}, Z_1 \in \calM_{n_2,K} \text{ with } \alpha=1+O\left(\sqrt{\frac{\log n_1}{n_1}}\right) \rbrace.
\end{align*}
 
 %Here $\calM'_{n_1,K}$ (resp. $\calM'_{n_2,K}$) is the subset of $\calM_{n_1,K}$ (resp $\calM_{n_2,K}$) such that Assumption \ref{ass:balanced_part} is satisfied with $\alpha=1+o(1)$. \hemant{I think the symbols $\calM'_{n_1,K}$, $\calM'_{n_2,K}$ do not need to be defined, we can add $\alpha = 1+o(1)$ in the definition of $\Theta$.}

 We want to lower bound $\inf_{\hat{z}} \sup_{\theta \in \Theta} \expec (r(\hat{z},z))$ where $\hat{z}$ is an estimator of $Z_1$. \rev{We will first present the core argument to obtain a lower bound for the special case $K=2$ (Theorem \ref{thm:minimax}). Then, we show that the the general case (where $K = L \geq 2$) can be handled by a reduction argument to the setting $K = 2$ (Theorem \ref{thm:minimax2}).}
 
 In the supervised case, i.e. when $Z_2$ is known, we can use the same strategy as the one use for the degree-corrected SBM by \cite{dcsbmgao2016} and obtain a lower bound of the order $e^{-(1+o(1))n_2(p-q)/2}$ corresponding to the failure probability of the optimal test associated to the following two hypothesis problem 
 \begin{align*}
     H_0:&  \otimes_{i=1}^{n_2/2} \mathcal{B}(p) \otimes_{i=n_2/2+1}^{n_2} \mathcal{B}(q), \text{ vs}\\
     H_1:&  \otimes_{i=1}^{n_2/2} \mathcal{B}(q) \otimes_{i=n_2/2+1}^{n_2} \mathcal{B}(p).
 \end{align*}  
 However when $n_2\gg n_1\log n_1$ the error associated with this hypothesis testing problem is of order $\exp(-n_2(p-q)/2) \approx \exp(-\sqrt{\frac{n_2\log n_1}{n_1}})$ but this is far smaller than $\exp(-n_1n_2p_{max}^2)$, the misclustering rate obtained for our algorithm. 
 %\hemant{This is not clear, needs to be written more clearly.}.
 A similar phenomena appears for high-dimensional Gaussian mixture models. Indeed, as shown in \cite{Ndaoud2018SharpOR}, it is essential to capture the hardness of estimating the model parameters in the minimax lower bound in order to get the right rate of convergence. The argument developed for obtaining a lower bound of the minimax risk in \cite{Ndaoud2018SharpOR} relies heavily on the Gaussian assumption (they set a Gaussian prior on the model parameters and use the fact that the posterior distribution is also Gaussian) and cannot directly be extended to this setting. 
 
% Our lower bound result is summarised in the following theorem.
 %
  \begin{theorem}\label{thm:minimax} 
  Suppose that $A \sim SBiSBM(Z_1, Z_2, p, q)$ with $K=L=2$,  $n_2\gg n_1 \log n_1$, $n_1n_2p_{max}^2\to \infty$ and $n_1n_2p_{max}^2=O( \log n_1$).
  % $n_1 n_2p_{max}^2 = \Theta(\log n_1)$, 
  %and $\sqrt{n_1\log n_1}n_2p_{max}^2\to 0$,
  Then there exists a constant $c_1 > 0$ such that \[ \inf_{\hat{z}} \sup_{\theta \in \Theta} \expec (r(\hat{z},z)) \geq \exp(-c_1n_1n_2p_{max}^2) \] where the infimum is taken over all measurable functions $\hat{z}$ of $A$. Moreover, if $n_1n_2p_{max}^2=\Theta(1)$, then $\inf_{\hat{z}} \sup_{\theta \in \Theta} \expec (r(\hat{z},z)) \geq c_2$ for some positive constant $c_2$. 
  %\gb{I wonder if we can extend the proof to show the following. In other words, $n_1n_2p_{max}^2\to \infty$ is a necessary condition for weak consistency.}
  
  %\hemant{ I am using $SBiSBM(Z_1, Z_2, p, q)$, is this consistent with earlier notation? Made edits in the statement of the theorem, check. Also, do you mean $n_1n_2p_{max}^2=O(1)$ or $n_1n_2p_{max}^2=\Theta(1)$? Need to be precise in the use of $O$ and $\Theta$ symbols everywhere. Also, don't we need some condition on $n_2 \gg n_1 \log n_1$ somewhere?}
\end{theorem}
\begin{remark}
This lower-bound shows that the rate of convergence of our estimator is optimal up to a constant factor. Indeed, for exact recovery, the minimax lower bound implies that $n_1n_2p_{max}^2 \gtrsim \log n_1$ is necessary, while for weak recovery, we need $n_1n_2p_{max}^2\to \infty$. It also shows that when $n_1n_2p_{max}^2=O(1)$ it is not possible to consistently estimating $Z_1$.
\end{remark}

%\hemant{I think we need to be precise in the use of asymptotic notation ($\Omega, O, \Theta$ etc.) and non-asymptotic notation ($\gtrsim, \lesssim, \asymp$) throughout the paper. }

%\begin{remark}
%For simplicity, we only write the proof for $K=L=2$ but we believe it can be extended to the case $ K = L>2$ by reducing the problem to test between two fixed communities as in \cite{dcsbmgao2016}. 
%\hemant{Is it straightforward, or do you ``believe'' it can be extended, as mentioned in the opening paragraphs? Need to be consistent with the language. And what about $K \neq L$?} 
%\end{remark}
\rev{In the general case where $K=L \geq 2$, we obtain by a reduction argument (to the case $K = 2$) the following theorem.}

 \begin{theorem}\label{thm:minimax2} 
 \rev{
  Suppose that $A \sim SBiSBM(Z_1, Z_2, p, q)$ with $K=L\geq 2$,  $n_2\gg n_1 \log n_1$, $n_1n_2p_{max}^2/(KL)\to \infty$ and $n_1n_2p_{max}^2/(KL)=O( \log n_1$).
  Then there exists a constant $c_1 > 0$ such that \[ \inf_{\hat{z}} \sup_{\theta \in \Theta} \expec (r(\hat{z},z)) \geq \exp(-c_1n_1n_2p_{max}^2/(KL)) \] where the infimum is taken over all measurable functions $\hat{z}$ of $A$. Moreover, if $n_1n_2p_{max}^2/(KL)=\Theta(1)$, then $\inf_{\hat{z}} \sup_{\theta \in \Theta} \expec (r(\hat{z},z)) \geq c_2/K^3$ for some positive constant $c_2$. }
\end{theorem}

\begin{remark}
\rev{
The minimax lower bound is of order $\exp(-\Theta(\frac{n_1n_2p^2}{KL}))$ and matches the upper bound $\exp(-O(\frac{n_1n_2p^2}{KL^2}))$ up to a $1/L$ factor. Extending our proof technique to the general non-symmetric case seems more challenging because the posterior distributions that appears in Step 2 (in the proof of Theorem \ref{thm:minimax}) have a far more complex expression.} 
%\hemant{Check this remark since upper bound might change now...}}
\end{remark}

\subsection{\rev{Proof of Theorem \ref{thm:minimax}}}

%\color{black}
The general idea of the proof is to first lower bound the minimax risk by an error occuring in a two hypothesis testing problem and then to replace this hypothesis testing problem by a simpler one. The steps are detailed below.

 \paragraph{Step 1.}
Recall that $z, z'$ denote the partition functions associated with $Z_1$ and $Z_2$ respectively.  We choose as a prior on $z$ and $z'$ a product of independent, centered Rademacher distributions. Since the marginals of $z,z'$ are sign invariant, then by using standard arguments, the results in \cite{dcsbmgao2016} (see the proof of Theorem 2) or \cite{Ndaoud2018SharpOR} show that (\rev{for any given $i=1,\dots,n_1$}) \[ \inf_{\hat{z}} \sup_{\theta \in \Theta} \expec (r(\hat{z},z)) \gtrsim \rev{ \inf_{\hat{z_i}} \expec_{z,z'} \expec_{A|z,z'}(\phi_i(A))}\] 
where $\hat{z}$ is a measurable function in $A$ and $\phi_i(A)=\indic_{\hat{z_i}\neq z_i}$. 
 
\paragraph{ Step 2.} 
% Without loss of generality, we can assume that $n_2$ is odd (else just delete one column: it won't increase the error).

We can write  \[ \expec_{z,z'} \expec_{A|z,z'}( \phi_i(A)) = \expec_{z_{-i}}\expec_{z_i} \expec_{\zz} \expec_{A|z,\zz}( \phi_i(A)) =\underbrace{ \expec_{z_{-i}}\expec_{z_i}  \expec_{A|z}( \phi_i(A))}_{R_i}\] where $z_{-i}:=(z_j)_{j\neq i}$ and $\expec_{X|Y}$ means that we integrate over the random variable $X$ conditioned on $Y$. Since $z'$ is random, the entries of $A$ are no longer independent (this is the reason why we use an uninformative prior on $\zz$). 
%\gb{I erased by mistake your comment here. As far as I remember you wanted more explanations for the following fact.}
Note that the quantity $\expec_{\zz} \expec_{A|z,\zz}( \phi_i(A))=\expec_{A|z}( \phi_i(A))$ only depends
%\footnote{Indeed, we have $\expec_{\zz}\expec[\phi_i(A)|z,\zz] = \expec[\phi_i(A)|z]$.}
on $f(A|z)$ the density of $A$ conditionally on $z$. Since the columns of $A$ are independent conditionally on $z$ (because $(\zz_{j})_{j=1}^{n_2}$ are independent) we have that $$f(A|z) = \prod_j f(A_{:j}|z),$$ i.e. $f(A|z)$ is the product of the densities of the columns $A_{:j}$ conditionally on $z$. Let us denote $A_{-ij}:=(A_{i'j})_{i'\neq i}$ and $A_{-i:}=(A_{i'j})_{i'\neq i, j}$. Now for each $j$ we can write
\[ f(A_{:j}|z) = f(A_{ij}|A_{-ij},z) f(A_{-ij}|z). \]
Since $A_{-ij}$ doesn't depend on $z_i$ we have $f(A_{-ij}|z)= f(A_{-ij}|z_{-i})$. Assume that $z_i=1$. Then $A_{ij} \sim \mathcal{B}(p)$ if $\zz_{j}=1$ or $A_{ij} \sim \mathcal{B}(q)$ if $\zz_{j}=-1$. 
This in turn implies 
\begin{align*} 
\prob (A_{ij}=1|A_{-ij},z)
&= p \prob(\zz_{j}=1 |A_{-ij},z )+q\prob(\zz_{j}=-1 |A_{-ij},z ) \\
&= p \prob(\zz_{j}=1 |A_{-ij},z_{-i} )+q\prob(\zz_{j}=-1 |A_{-ij}, z_{-i})
\end{align*}
Let us denote by $\alpha_j$ the random variable $\prob(\zz_{j}=1 |A_{-ij},z_{-i} )$. When $z_i = -1$, similar considerations imply 
\begin{align*} 
\prob (A_{ij}=1|A_{-ij},z_{-i})
= q \alpha_j + p (1-\alpha_j).
\end{align*}
This shows that the conditional distribution  \begin{equation}\label{eq:cond_dist1}
    A_{ij}|A_{-ij},z \sim \calB(\alpha_jp+(1-\alpha_j)q) \quad \text{when } z_i=1 
\end{equation}  and  \begin{equation}\label{eq:cond_dist2}
    A_{ij}|A_{-ij},z \sim \calB(\alpha_jq+(1-\alpha_j)p) \quad \text{when } z_i=-1. 
\end{equation}
We can write $R_i$ as \begin{align*}
   R_i &= \expec_{z_{-i}}\expec_{z_i}  \expec_{A_{-i:}|z}\expec_{A_{i:}|z,A_{-i:}}( \phi_i(A))\\
   &=\expec_{z_{-i}}\expec_{z_i}  \expec_{A_{-i:}|z_{-i}}\expec_{A_{i:}|z,A_{-i:}}( \phi_i(A))\\
   &=\expec_{z_{-i}}  \expec_{A_{-i:}|z_{-i}}\underbrace{\expec_{z_i}\expec_{A_{i:}|z,A_{-i:}}( \phi_i(A))}_{R'_i}.
\end{align*}

The term $R_i'$ corresponds to the risk
associated with the following two hypothesis testing problem (conditionally on $(A_{i'j})_{i'\neq i}$ and $z_{-i}$)
\begin{equation}\label{eq:minimax_test_hyp1}
     H_0:  \otimes_{j=1}^{n_2} \mathcal{B}(\alpha_jp+(1-\alpha_j)q) \text{ vs }
     H_1:  \otimes_{j=1}^{n_2} \mathcal{B}(\alpha_jq+(1-\alpha_j)p).
\end{equation}  
%\gb{Rewrite that}
%
%\hemant{In the above test, shouldn't the factor $\prod_j f((A_{i'j})_{i'\neq i}|z_{-i})$ appear for both $H_0, H_1$? If not, why?}
%\gb{ This is a consequence of Neyman Pearson lemma: the minimum error risk is attained for the likelihood ratio statistic and the common factor is cancelled out, so the error is the same as the one associated with the test between $H_0'$ and $H_1'$. }

%Unfortunately, for this testing problem, it is difficult to obtain a simplified expression of the optimal test given by the Neyman-Pearson lemma.
Our goal is to replace this two-hypothesis testing problem by a simpler hypothesis test associated with a smaller error. The proof strategy is the following. When $\alpha_j$ is very close to $1/2$ it is not possible to statistically distinguish $(\mathcal{B}(\alpha_jp+(1-\alpha_j)q)$ from $(\mathcal{B}(\alpha_jq+(1-\alpha_j)p)$ and these factors can be dropped. When $\alpha_j$ is significantly different from $1/2$, then the risk associated to the test can be lower bounded by a the risk of testing between  a product of $\mathcal{B}(p)$ vs. a product of $\mathcal{B}(q)$. More precisely, we will show that the number of indices $j$ for which $\alpha_j$ is significantly different from $1/2$ is of order $n_1n_2p$ and hence the risk is lower bounded by the one associated by the two hypothesis testing problem
%When $\alpha_j>1/2$ we can replace the testing problem $(\alpha_j\mathcal{B}(p)+(1-\alpha_j)\mathcal{B}(q))$ vs $(\alpha_j\mathcal{B}(q)+(1-\alpha_j)\mathcal{B}(p))$ by the simpler testing problem $\mathcal{B}(p)$ vs $\mathcal{B}(q)$ and do a similar simplification when $\alpha_j <1/2$. It remains to control the number of indices for which $\alpha_j>1/2$ and $\alpha_j < 1/2$. It will be shown that they both have the same order $n_2 n_1 p$. Consequently, we would have obtained \[ \inf_\psi \prob_{H_0} \psi +\prob_{H_1} (1-\psi) \gtrsim \inf_\psi \prob_{H'_0} \psi +\prob_{H'_1} (1-\psi)\] where the infimum is taken over all functions $\psi$ measurable in $A$ and $H'_0, H'_1$ correspond to the ``simpler'' tests 
%
\begin{align*}
     H'_0:  \otimes_{i=1}^{n_2n_1p} \mathcal{B}(p)\otimes_{i=1}^{n_2n_1p} \mathcal{B}(q)  \text{ vs }
     H'_1:  \otimes_{i=1}^{n_2n_1p} \mathcal{B}(q)\otimes_{i=1}^{n_2n_1p} \mathcal{B}(p).
 \end{align*} 
It is well known that the error associated with the above testing problem is of the order $e^{-\Theta(n_2n_1p^2)}$ (see e.g. \cite{dcsbmgao2016}) 
which corresponds to the rate of convergence of our algorithm. Let us now formalize the above argument.

\paragraph{Step 3.} 
First, let us define 
\begin{align*}
    \calC_+&=\lbrace i'\neq i: z_{i'}=1\rbrace,\\
    \calC_-&=\lbrace i'\neq i: z_{i'}=-1\rbrace,\\
    \theta_j&=\frac{\alpha_j}{1-\alpha_j}\text{ for } j=1\ldots n_2,\\
    \epsilon&= \Theta(p\sqrt{n_1\log n_1})=o(1),\\
    J_b&=\lbrace j\in [n_2] : \theta_j\in [1-\epsilon, 1+\epsilon]\rbrace,\\ 
    J_g&=\lbrace j\in [n_2] : \theta_j\notin [1-\epsilon, 1+\epsilon]\rbrace,\\
    T_j&\overset{ind.}{\sim}\calB(\alpha_j), \text{ for all } j\in J_g,
\end{align*}
 and the events
 \begin{align*}
      \calE_1 &=\left\lbrace  |\calC_+|-|\calC_-|\in [-C\sqrt{n_1\log n_1}, C\sqrt{n_1 \log n_1}]\right\rbrace \text{ for some constant }C>0,\\
       \calE_2 &=\left\lbrace  \sum_j \indic_{\lbrace \sum_{i'\in \calC_+}A_{i'j}-\sum_{i'\in \calC_-}A_{i'j}\neq 0 \rbrace} = \Theta( n_2n_1p) \right\rbrace,\\
        \calE_3&=\left\lbrace \sum_{j\in J_b}A_{ij} = \Theta(n_2p) \right\rbrace .
 \end{align*}
We will show later that these events occur with high probability. They are useful for obtaining a lower bound on the densities $f(A_{ij}|A_{-ij},z)$.
Since we are integrating over positive functions one can write \begin{equation}\label{eq:minimax_cond1}
     R_i\geq \expec_{z_{-i}} \indic_{\calE_1} \expec_{A_{-i:}|z_i} \indic_{\calE_2}\expec_{z_i}\expec_{A_{i:}|z,A_{-i:}}( \indic_{\calE_3} \phi_i(A)).
\end{equation} 

\paragraph{Step 4.} For all $j\in J_b$, the set for which $\alpha_j\approx \frac{1}{2}$, we are going to lower bound the densities $f(A_{ij}|A_{-ij},z)$ by $g(A_{ij})$ corresponding to the density of $\mathcal{B}(\frac{p+q}{2})$. 
A simple calculation shows that \[ \theta_j=\frac{\alpha_j}{1-\alpha_j}= \frac{\prod_{i' \neq i: z_{i'}=1}p^{A_{i'j}}(1-p)^{1-A_{i'j}} \prod_{i' \neq i: z_{i'}=-1}q^{A_{i'j}}(1-q)^{1-A_{i'j}} }{\prod_{i' \neq i: z_{i'}=1}q^{A_{i'j}}(1-q)^{1-A_{i'j}} \prod_{i' \neq i: z_{i'}=-1}p^{A_{i'j}}(1-p)^{1-A_{i'j}}}.\]
The previous expression can be rewritten as \[ \theta_j = \left(\frac{p(1-q)}{q(1-p)}\right)^{\sum_{i'\in \calC_+}A_{i'j}-\sum_{i'\in \calC_-}A_{i'j}}\left(\frac{1-p}{1-q}\right)^{|\calC_+|-|\calC_-|} .\]  We also have the relation $\alpha_j = \frac{\theta_j}{1+\theta_j}$, so $\alpha_j$ is close to $1/2$ if and only if $\theta_j$ is close to $1$. If $z_{-i}$ were an exactly balanced partition, i.e., $|\calC_+|-|\calC_-|=0$, then $\theta_j=1$ would be equivalent to $\sum_{i'\in \calC_+}A_{i'j}-\sum_{i'\in \calC_-}A_{i'j}=0$. 
However the contribution of the term $\left(\frac{1-p}{1-q}\right)^{|\calC_+|-|\calC_-|}$ is small under $\calE_1$. Indeed, we have $|\calC_+|-|\calC_-|\in [-C\sqrt{n_1\log n_1}, C\sqrt{n_1 \log n_1}]$.
Note that $\log (\frac{1-p}{1-q})\in [-c(p-q),c(p-q)]$ by using Taylor's formula for some constant $c>0$. Hence, under $\calE_1$ 
\[  
\left(\frac{1-p}{1-q}\right)^{|\calC_+|-|\calC_-|} \in [e^{-c'\sqrt{n_1\log n_1}p},e^{c'\sqrt{n_1\log n_1}p} ]
\] 
for some constant $c'>0$. But we have \[ \max (|e^{c'p\sqrt{n_1\log n_1}}-1|,|e^{-c'p\sqrt{n_1\log n_1}}-1|)\leq c'p\sqrt{n_1\log n_1}:= \epsilon .\] Therefore, it follows that  $(\frac{1-p}{1-q})^{|\calC_+|-|\calC_-|} \in [1-\epsilon , 1+ \epsilon]$ under $\calE_1$.
It is easy to check  $\theta_j\in [1-\epsilon, 1+\epsilon]$ implies $\alpha_j\in [1/2-\epsilon', 1/2+\epsilon']$ for $\epsilon'$ proportional to $\epsilon$. Since the constant involved here doesn't matter, we won't make a distinction between $\epsilon$ and $\epsilon'$. 

Now recall that by \eqref{eq:cond_dist1} and \eqref{eq:cond_dist2} we have \[ f(A_{ij}|A_{-ij},z)=\alpha_jp^{A_{ij}}(1-p)^{1-A_{ij}}+(1-\alpha_j)q^{A_{ij}}(1-q)^{1-A_{ij}} \text{ when } z_i=1\] and  \[ f(A_{ij}|A_{-ij},z)=\alpha_jq^{A_{ij}}(1-q)^{1-A_{ij}}+(1-\alpha_j)p^{A_{ij}}(1-p)^{1-A_{ij}} \text{ when } z_i=-1.\]
\begin{itemize}
    \item When $z_i=1$, we have for  all $j\in [n_2]$ such that $\alpha_j\in [1/2-\epsilon, 1/2+\epsilon]$ that \[ \alpha_jp^{A_{ij}}(1-p)^{1-A_{ij}}+(1-\alpha_j)q^{A_{ij}}(1-q)^{1-A_{ij}}\geq (1-\epsilon) \frac{p+q}{2}\indic_{A_{ij}=1} + (1-p\epsilon)\frac{1-p+1-q}{2}\indic_{A_{ij}=0} \] and \scriptsize \[ \prod_{j \in J_b}\alpha_j p^{A_{ij}}(1-p)^{1-A_{ij}}+(1-\alpha_j)q^{A_{ij}}(1-q)^{1-A_{ij}} \geq (1-\epsilon)^{\sum_{j\in J_b}A_{ij}} (1-p\epsilon)^{\sum_{j\in J_b}(1-A_{ij})}\prod_{j \in J_b}\frac{p^{A_{ij}}(1-p)^{1-A_{ij}}+q^{A_{ij}}(1-q)^{1-A_{ij}} }{2}. \]
    \normalsize
    \item  When $z_i=-1$, we have for  all $j\in [n_2]$ such that $\alpha_j\in [1/2-\epsilon, 1/2+\epsilon]$ \[ \alpha_jq^{A_{ij}}(1-q)^{1-A_{ij}}+(1-\alpha_j)p^{A_{ij}}(1-p)^{1-A_{ij}}\geq (1-\epsilon) \frac{p+q}{2}\indic_{A_{ij}=1} + (1-q\epsilon)\frac{1-p+1-q}{2}\indic_{A_{ij}=0}\] and \scriptsize \[ \prod_{j \in J_b}\alpha_j q^{A_{ij}}(1-q)^{1-A_{ij}}+(1-\alpha_j)p^{A_{ij}}(1-p)^{1-A_{ij}} \geq (1-\epsilon)^{\sum_{j\in J_b}A_{ij}} (1-q\epsilon)^{\sum_{j\in J_b}(1-A_{ij})}\prod_{j \in J_b}\frac{p^{A_{ij}}(1-p)^{1-A_{ij}}+q^{A_{ij}}(1-q)^{1-A_{ij}} }{2}. \]
\normalsize
\end{itemize}
Since on $\calE_3$, $\sum_{j\in J_b}A_{ij}=\Theta(n_2p)$ we obtain that \[ (1-\epsilon)^{\sum_{j\in J_b}A_{ij}}\geq 1-n_2\epsilon p=1-o(1)  \]  and similarly \[ (1-p\epsilon)^{\sum_{j\in J_b}(1-A_{ij})}\geq 1-o(1)\text{ and } (1-q\epsilon)^{\sum_{j\in J_b}(1-A_{ij})}\geq 1-o(1).\]
This implies the lower bound \[ \indic_{\calE_3}\prod_{j\in J_b}f(A_{ij}|A_{-ij},z)\geq (1-o(1))\indic_{\calE_3} \prod_{j\in J_b}g(A_{ij}).\]
Consequently, we obtain the lower bound \begin{align*}
    \expec_{z_i}&\expec_{A_{i:}|z,A_{-i:}}( \indic_{\calE_3} \phi_i(A))\\
    &\geq (1-o(1))\int_{z_i}\int_{A_{i:}}\indic_{\calE_3} \phi_i(A) \prod_{j\in J_b}g(A_{ij})\prod_{j\in J_g}f(A_{ij}|A_{-ij},z)\mathrm{d}A_{i:}\mathrm{d}\prob( z_{i}) \\
     &\gtrsim \int_{z_i}\int_{(A_{ij})_{j\in J_b}}\int_{(A_{ij})_{j\in J_g}}\prod_{j\in J_b}g(A_{ij}) \left( \phi_i(A) \prod_{j\in J_g}f(A_{ij}|A_{-ij},z)\mathrm{d}(A_{ij})_{j\in J_g} \right) \mathrm{d}(A_{ij})_{j\in J_b}\mathrm{d}\prob( z_{i})\\
    &\quad -\tilde{\prob}_{A_{i:}|A_{-i:},z}(\calE_3^c)
    \\
      &\gtrsim \int_{z_i}\int_{(A_{ij})_{j\in J_b}}\prod_{j\in J_b}g(A_{ij}) \int_{(A_{ij})_{j\in J_g}} \left( \phi_i(A) \prod_{j\in J_g}f(A_{ij}|A_{-ij},z)\mathrm{d}(A_{ij})_{j\in J_g} \right) \mathrm{d}(A_{ij})_{j\in J_b}\mathrm{d}\prob( z_{i})\\
      &\quad -\tilde{\prob}_{A_{i:}|A_{-i:},z}(\calE_3^c)\tag{since $J_b$ and $J_g$ are disjoint}
    \\
    &\gtrsim \int_{(A_{ij})_{j\in J_b}}\prod_{j\in J_b}g(A_{ij}) \underbrace{\left(\int_{z_i} \int_{(A_{ij})_{j\in J_g}}\phi_i(A) \prod_{j\in J_g}f(A_{ij}|A_{-ij},z)\mathrm{d}(A_{ij})_{j\in J_g} \right)}_{R''_i} \mathrm{d}(A_{ij})_{j\in J_b}\mathrm{d}\prob( z_{i})\\
    &\quad -\tilde{\prob}_{A_{i:}|A_{-i:},z}(\calE_3^c) \tag{since $g(A_{ij})$ is independent from $z_i$}
\end{align*}  
where $\tilde{\prob}_{A_{i:}|A_{-i:},z}$ is the conditional distribution corresponding to the density product \[\prod_{j\in J_b}g(A_{ij})\prod_{j\in J_g}f(A_{ij}|A_{-ij},z)\mathrm{d}\prob( z_{i}).\]

\paragraph{Step 5.} $R''_i$ corresponds to risk associated to the testing problem \eqref{eq:minimax_test_hyp1} where the product is restricted to the set $J_g$. One can lower bound this risk as follows. 

By definition of $T_j$, we have  when $z_i=1$ by equation \eqref{eq:cond_dist1} \[ f(A_{ij}|A_{-ij},z)= \expec_{T_j}\left(T_jp^{A_{ij}}(1-p)^{1-A_{ij}}+ (1-T_j)q^{A_{ij}}(1-q)^{1-A_{ij}}\right):=\expec_{T_j}\left(f_{T_j}(A_{ij})\right). \] A similar result holds when $z_i=-1$.
Since $T_j$ is independent of $z_i$ (it only depends on $A_{-ij}$ and $z_{-i}$ through $\alpha_j$) we can write
\[ R''_i=\expec_{T}\expec_{z_i}\underbrace{\int_{(A_{ij})_{j\in J_g}}\phi_i(A) \prod_{j\in J_g}f_{T_j}(A_{ij})\mathrm{d}(A_{ij})_{j\in J_g}}_{R'''_i}=\expec_{T}(R_i''')\] where $T=(T_j)_{j\in J_g}$.
But now $R'''_i$ is the risk associated with the following two hypothesis testing problem
\begin{align*}
     H''_0: &  \otimes_{T_j=1} \mathcal{B}(p)\otimes_{T_j=0} \mathcal{B}(q) , \text{ vs}\\
     H''_1:&  \otimes_{T_j=1} \mathcal{B}(q)\otimes_{T_j=0} \mathcal{B}(p).
 \end{align*} 
The error associated with this test is lower bounded by the error associated with the test
\begin{align*}
     H'''_0: &  \otimes_{j\in J_g} \mathcal{B}(p)\otimes_{j\in J_g} \mathcal{B}(q) , \text{ vs}\\
     H'''_1:&  \otimes_{j\in J_g} \mathcal{B}(q)\otimes_{j\in J_g} \mathcal{B}(p).
 \end{align*} 
since adding more information can only decreases the error  (more formally, the lower bound follows from the fact that we multiply by density functions inferior to one). Under $\calE_2\cap \calE_1$ we have \begin{equation}
    \label{eq:jb}
    |J_g|=\Theta(n_1n_2p)\quad \text{ and }  |J_b|=(1-o(1))n_2.
\end{equation}  
Hence, by Lemma 4 in \cite{dcsbmgao2016} the risk is lower bounded by $R_i''\rev{\geq} e^{-\Theta(n_1n_2p^2)}$. 

\paragraph{Conclusion.}  It remains to integrate over all the events we conditioned on. We have so far shown that \begin{align}
    R_i& \geq \expec_{z_{-i}}\indic_{\calE_1}\expec_{A_{-i:}|z_{-i}}\indic_{\calE_2} e^{-\Theta(n_1n_2p^2)}-\expec_{z_{-i}}\indic_{\calE_1}\expec_{A_{-i:}|z_{-i}}\indic_{\calE_2}\tilde{\prob}_{A_{i:}|A_{-i:},z}(\calE_3^c)\nonumber\\
    &\geq e^{-\Theta(n_1n_2p^2)} \expec_{z_{-i}}(\indic_{\calE_1} \prob_{A_{-i:}|z_{-i}}(\calE_2))-\expec_{z_{-i}}\indic_{\calE_1}\expec_{A_{-i:}|z_{-i}}\indic_{\calE_2}\tilde{\prob}_{A_{i:}|A_{-i:},z}(\calE_3^c).\label{eq:minimax_concl}
\end{align}
Now these probabilities can be controlled with the following lemma proved in the appendix, see Section \ref{app:proof_lem_1}.
\begin{lemma}\label{lem:cond_prob_minimax} We have 
\begin{enumerate}
    \item \ $\prob_{z_{-i}}(\calE_1)\geq 1-n_1^{-\Omega(1)}$ ;
    \item \ for all realizations $z_{-i} $, $\prob_{A_{-i:}|z_{-i}}(\calE_2)\geq 1-e^{-\Omega(n_1n_2p)}$;
    \item \ for all $z_{-i}\in \calE_1$ and $A_{-i:}\in \calE_2$, $\tilde{\prob}_{A_{i:}|A_{-i:},z}(\calE_3)\geq 1-e^{-\Omega(n_2p)}$.
\end{enumerate}
\end{lemma}
Lemma \ref{lem:cond_prob_minimax} and \eqref{eq:minimax_concl} directly imply  $R_i\gtrsim e^{-\Theta(n_1n_2p^2)}$ since $e^{-\Theta(n_2p)}=o(e^{-\Theta(n_1n_2p^2)})$ under the assumption on $p$.

When $n_1n_2p_{max}^2=\Theta (1)$, we can use the same proof except that now the final two-hypothesis testing problem we reduced to has a non vanishing risk (see e.g. \cite{minimaxsbm}).

\subsection{\rev{Proof of Theorem \ref{thm:minimax2}}}
%\color{blue}
The main idea of the proof is to reduce the problem to the case $K=L=2$, after which we can proceed along the lines of the proof of Theorem \ref{thm:minimax}. The steps are outlined below. 
\paragraph{Step 1.} Instead of considering independent vectors with marginals given by independent Rademacher random variables, we will use $z,z'$ that are random vectors with independent marginals following the uniform law on $[K]$. The argument from \cite{dcsbmgao2016} used in the proof of Theorem \ref{thm:minimax} leads to the following lower bound (for any given $i=1,\dots,n_1$)
\[ \inf_{\hat{z}} \sup_{\theta \in \Theta} \expec (r(\hat{z},z)) \gtrsim  \frac{1}{K^3}\inf_{\hat{z_i}} \expec_{z,z'} \expec_{A|z,z'}(\phi_i(A))= \frac{1}{K^3} \inf_{\hat{z_i}} \expec_{z} \expec_{A|z}(\phi_i(A)).\]
%
%\frac{K}{n_1}\sum_{i=1}^{\ceil{n_1/K}}
\paragraph{Step 2.}
Let us fix $k\neq k'\in [K]$, we then have 
\[ \phi_i(A) = \indic_{\hat{z_i}\neq z_i} \geq \tilde{\phi}_i(A):=\indic_{\lbrace \hat{z}_i=k, z_i=k'\rbrace}+\indic_{\lbrace\hat{z}_i=k', z_i=k\rbrace}.\] 
So if $z_i\not \in \lbrace k,k'\rbrace$ then $\tilde{\phi}_i(A)=0$ whatever the choice of the estimator $\hat{z_i}$. 
%
%Let $\tilde{z}_i$ be a r.v. that takes value $k$ or $k'$ with probability $0.5$ and denote $\tilde{z}$ the random variable obtained from $z$ where the $i$-th entry of $z$ is replaced by $\tilde{z}_i$. 
%
%
Denoting $\calS$ to be the event $ \left\lbrace z_i \in \lbrace k , k' \rbrace \right\rbrace $, we then have  
\begin{align*}
    \expec_{z_i}  \expec_{A|z}( \phi_i(A)) &\geq  \expec_{z_i}\expec_{A|z}( \tilde{\phi}_i(A))\\
    &=\prob(\calS) \expec_{z_i}\left(  \expec_{A|z}( \tilde{\phi}_i(A))|\calS \right)\\
    &=\frac{2}{K}\expec_{z_i}\left(  \expec_{A|z}( \tilde{\phi}_i(A))|\calS \right).
\end{align*}  
%
%
%where $\tilde{z}_i$ is a r.v. that takes value $k$ or $k'$ with probability $0.5$. %\gb{We should replace $z$ by $\tilde{z}$ where $(\tilde{z})_j=z_j$ when $j\neq i$ and $\tilde{z}_i$ when $i=j$. }
% 
Now let us define $z''=h(z')$ as follows. If $z'_j\in  [K]\setminus \lbrace k,k'\rbrace$ then $z''_j=z'_j$. Otherwise $z''_j=*$ where ``$*$'' is any symbol that does not belong to $[K]$. In other words, $z''$ only keeps the information about the column labels that are different from $k$ and $k'$.
 %Now, observe that $z'$ can be generated in the following way. First, for each entry $j$, $z''_j=l$ for $l\in [K]\setminus \lbrace k,k'\rbrace$ with probability $1/K$ or $z''_j=*$ with probability $2/K$. Then we replace the entries $z''_j=*$ by $k$ or $k'$ with probability $1/2$ to obtain $z'$.
%
%
\begin{remark}
 Note that the point of the above construction is to reduce our problem to the case where $K = L = 2$. First, we reduce to the case where $z_i$ can take only two values, namely $k$ or $k'$. Then we `display' only the column labels which are different from $k$ and $k'$ -- those belonging to $\set{k,k'}$ are masked by the symbol `$*$'. Indeed, the columns with labels different from $k, k'$ do not provide any information for distinguishing $z_i = k$ from $z_i = k'$.
\end{remark}
Since $z''$ is independent of $z$, we obtain\footnote{By definition the distribution of $A|z,z''_j=*$ is the same as $A|z, z'_j\in \lbrace k,k' \rbrace $ and the distribution of $A|z,z''_j=l$ is the same as $A|z, z'_j=l \not\in \lbrace k,k' \rbrace $} \[ \expec_{z_i}\left(  \expec_{A|z}( \tilde{\phi}_i(A)) |\calS \right)=\expec_{z_i}\left(\expec_{z''}  \expec_{A|z,z''}( \tilde{\phi}_i(A))|\calS \right)=\expec_{z''} \expec_{z_i}\left( \expec_{A|z,z''}( \tilde{\phi}_i(A))|\calS \right).\] 
 
Conditionally on $z$, $\calS,z''$, note that the columns of $A$ are independent since $(z'_{j})_{j=1}^{n_2}$ are independent, hence $$f(A|z, \calS, z'') = \prod_j f(A_{:j}|z, \calS, z_j''),$$ i.e. $f(A|z, \calS, z'')$ is the product of the densities of the columns $A_{:j}$ conditionally on $z, \calS,$ and $z''$. 
%
%\hemant{The previous sentence seems strange. If you condition on $z,z''$, then the entries of A are independent by the model assumption, so what does it have to do with $(z''_{j})_{j=1}^{n_2}$ being independent?}
%\gb{The entries of $A$ are independent conditionally on $z$ and $z'$. But when you only condition on $z''$ you create decencies between the entries in each columns, but the columns remain independent. For example, if $L=K=2$, $z''$ doesn't contain any information and it is the same argument as we used for Theorem 3.}
%
For each $j$ we can write
\[ f(A_{:j}|z, \calS,z_j'') = f(A_{ij}|A_{-ij},z, \calS, z_j'') f(A_{-ij}|z, \calS, z_j'')=f(A_{ij}|A_{-ij},z, \calS, z_j'') f(A_{-ij}|z_{-i},z_j'') \] and $f(A_{-ij}|z_{-i},z_j'')$ doesn't depends on $z_i$. We have the following scenarios depending on the value of $z_i$.
\begin{enumerate}
\item \ \underline{Assume that $z_i=k$}. Then $A_{ij} \sim \mathcal{B}(p)$ if $\zz_{j}=k$ or $A_{ij} \sim \mathcal{B}(q)$ if $\zz_{j}\neq k$. \begin{itemize}
    \item If $z_j''=*$, we have
        \begin{align*} 
            \prob (A_{ij}=1|A_{-ij},z, \calS, z_j''=*)
            &= p \prob(\zz_{j}=k |A_{-ij},z, \calS, z''_j=* )+q \prob(\zz_{j}=k' |A_{-ij},z, \calS, z''_j=*  )\\
            &= p \underbrace{\prob(\zz_{j}=k |A_{-ij},z_{-i},z''_j=* )}_{\alpha_j}+q\underbrace{\prob(\zz_{j}=k' |A_{-ij},z_{-i},z''_j=* ) }_{1-\alpha_j}.
        \end{align*}
    \item If $z_j''\neq *$, then 
    $$\prob (A_{ij}=1|A_{-ij},z, \calS, z_j'')=q.$$
\end{itemize}

\item \ \underline{Assume that $z_i=k'$.} 
\begin{itemize}
    \item If $z_j''=*$, similar considerations as before show that 
    \[ \prob (A_{ij}=1|A_{-ij},z, \calS, z_j''=*)= q \alpha_j+p(1-\alpha_j). \]
    
    \item If $z_j''\neq *$, then $$\prob (A_{ij}=1|A_{-ij},z, \calS, z_j'')=q.$$
\end{itemize}
\end{enumerate}
%In particular if $z''\neq *$, $\prob (A_{ij}=1|A_{-ij},z, \calS, z_j'')=q$ whatever the value of $z_i\in \lbrace k, k' \rbrace$.

%Similarly, when $\tilde{z}_i=k'$, similar considerations imply \begin{align*} \prob (A_{ij}=1|A_{-ij},z, \calS, z_j''=*)&= q \alpha_j+p(1-\alpha_j).\end{align*}
%
Denote $\calJ=\lbrace j\in [n_2]: z''_j=*\rbrace$. Then by a similar argument as in the proof of Theorem \ref{thm:minimax} we can show that we have obtained a lower bound corresponding to the risk between the two-hypothesis testing problem (conditionally on $(A_{i'j})_{i'\neq i}$, $z_{-i}$ and $z''$)
\begin{equation}\label{eq:minimax_test_hyp1b}
     H_0:  \otimes_{j\in \calJ} \mathcal{B}(\alpha_jp+(1-\alpha_j)q)\otimes_{j\in \calJ^c} \mathcal{B}(q)\text{ vs }
     H_1:  \otimes_{j\in \calJ} \mathcal{B}(\alpha_jq+(1-\alpha_j)p)\otimes_{j\in \calJ^c}\mathcal{B}(q).
\end{equation} 
This is equivalent to
\begin{equation}\label{eq:minimax_test_hyp1bb}
     H_0:  \otimes_{j\in \calJ} \mathcal{B}(\alpha_jp+(1-\alpha_j)q) \text{ vs }
     H_1:  \otimes_{j\in \calJ} \mathcal{B}(\alpha_jq+(1-\alpha_j)p).
\end{equation}  
By using a conditioning argument on an event that occurs with high probability as in the proof of Theorem \ref{thm:minimax}, one can assume that $|\calJ|=(2+o(1))n_2/K$. We have reduced to the problem of testing whether $\tilde{z_i}=k$ or $k'$ with the knowledge of all the values of $z'$ that are different from $k$ and $k'$.
Then, to obtain the stated result it remains to apply the proof of Theorem \ref{thm:minimax} from Step 3 with \begin{align*}
    \calC_+&=\lbrace i'\neq i: z_{i'}=k\rbrace,\\
    \calC_-&=\lbrace i'\neq i: z_{i'}=k'\rbrace,\\
    \theta_j&=\frac{\alpha_j}{1-\alpha_j}\text{ for } j=1\ldots n_2,\\
    \epsilon&= \Theta(p\sqrt{n_1\log n_1})=o(1),\\
    J_b&=\lbrace j\in \calJ : \theta_j\in [1-\epsilon, 1+\epsilon]\rbrace,\\ 
    J_g&=\lbrace j\in \calJ : \theta_j\notin [1-\epsilon, 1+\epsilon]\rbrace,\\
    T_j&\overset{ind.}{\sim}\calB(\alpha_j), \text{ for all } j\in J_g,
\end{align*}
 and the events
 \begin{align*}
      \calE_1 &=\left\lbrace  |\calC_+|-|\calC_-|\in [-C\sqrt{n_1/K\log (n_1/K)}, C\sqrt{n_1/K\log (n_1/K)}]\right\rbrace \text{ for some constant }C>0,\\
       \calE_2 &=\left\lbrace  \sum_j \indic_{\lbrace \sum_{i'\in \calC_+}A_{i'j}-\sum_{i'\in \calC_-}A_{i'j}\neq 0 \rbrace} = \Theta( n_2n_1p/K^2) \right\rbrace,\\
        \calE_3&=\left\lbrace \sum_{j\in J_b}A_{ij} = \Theta(n_2p/K) \right\rbrace .
 \end{align*}
 
 \paragraph{Steps 4 and 5.} It is easy to check that \[ \theta_j=\frac{\alpha_j}{1-\alpha_j}= \frac{\prod_{i' \neq i: z_{i'}=k}p^{A_{i'j}}(1-p)^{1-A_{i'j}} \prod_{i' \neq i: z_{i'}=k'}q^{A_{i'j}}(1-q)^{1-A_{i'j}} }{\prod_{i' \neq i: z_{i'}=k}q^{A_{i'j}}(1-q)^{1-A_{i'j}} \prod_{i' \neq i: z_{i'}=k'}p^{A_{i'j}}(1-p)^{1-A_{i'j}}}.\] By using the same argument as in the proof of Theorem \ref{thm:minimax} the error associated with the test \eqref{eq:minimax_test_hyp1bb} can be lower bounded by
 \begin{equation}\label{eq:minimax_test_hyp2bb}
     H_0:  \otimes_{j\in \calJ_g } \mathcal{B}(p) \text{ vs }
     H_1:  \otimes_{j\in \calJ_g } \mathcal{B}(q).
\end{equation}  
We can conclude by using a similar conditioning argument so that $|\calJ_g|=\Theta(\frac{n_1n_2p^2}{K^2})$ and the risk of the test \eqref{eq:minimax_test_hyp2bb} is lower bounded by $e^{-\Theta(\frac{n_1n_2p^2}{K^2})}$. The stated result is obtained by using the same steps as in Theorem \ref{thm:minimax} and multiplying by $2/K$ (see Step 2). This factor could possibly be removed by lower-bounding $\phi_i(A)$ by the sum over all $k'\neq k$ of functions of the form $\tilde{\phi}_i(A)$. But when $\frac{n_1n_2p^2}{K^2}\to \infty$, the factor $2/K$ can be absorbed by the $\Theta(\frac{n_1n_2p^2}{K^2})$. 
%\hemant{what about the $1/K^3$ factor in Step 1??} \gb{I corrected the statement of Theorem 4}. 
This improvement is only interesting when $\frac{n_1n_2p^2}{K^2}=\Theta(1)$.
\color{black}
\section{Numerical experiments}\label{sec:xp}
In this section, we empirically compare the performance of our algorithm (\gpm) with the spectral algorithm (\spec)  and the algorithm introduced by \cite{Ndaoud2021ImprovedCA} (referred to as \hl). 

\paragraph{Case $K=L=2$.} In this setting, we generate a SBiSBM with parameters $n_1=500$, $n_2=\ceil{Cn_1\log n_1}$ (where $C\geq 1$ is a constant), $p \in (0,1)$, and $q=c p$ where $c>0$ is a constant. 
%\hemant{I think $\delta$ should be replaced with $c$ since we use $q = cp$ in the text. it would need to be changed in the caption too..}
The accuracy of the clustering is measured by the Normalized Mutual Information (NMI); it is equal to one when the partitions match exactly and is zero when they are independent. The results are averaged over $20$ Monte-Carlo runs.
For the experiment presented in Figure \ref{fig:xp1}, we fixed $C=10$; for the experiment in Figure \ref{fig:xp2}, we fixed $C=3$ and $c=0.5$. For the experiment presented in Figure \ref{fig:xp3}, we fixed $p=0.01$ and $c=0.5$. We observe that \hl\, and \gpm\, have similar performance in all the aforementioned experiments. Thus, there is no gain in using the specialized method \hl\, instead of the general algorithm \gpm. The spectral method \spec has only slightly worst performance than the iterative methods \hl\, and \gpm. In particular, when approaching the threshold for exact recovery, the performance gap disappears. This suggests that \spec\, also reaches the threshold for exact recovery. It would be interesting to obtain stronger theoretical guarantees to explain the good performance of \spec.

 \begin{figure*}
        \centering
        \begin{subfigure}[b]{0.475\textwidth}
        \centering
         \includegraphics[width=\textwidth]{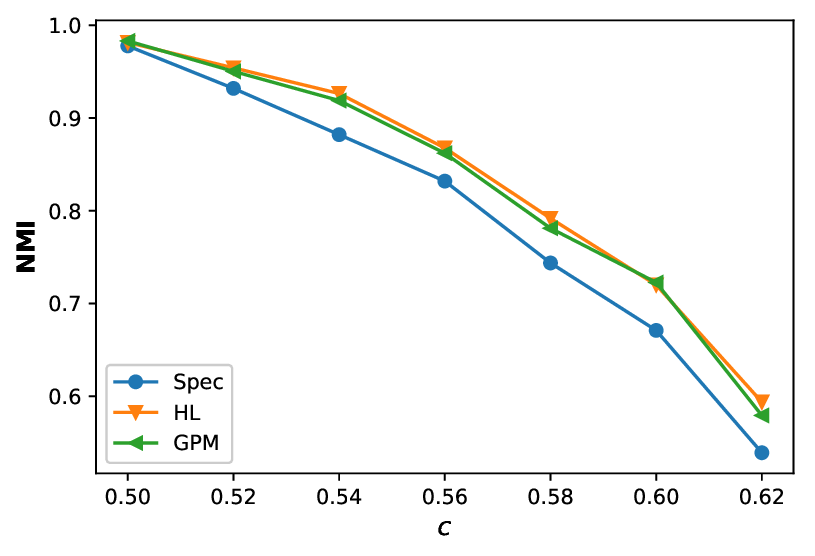}
            \caption[Network2]%
            {{\small Relative performance of \spec, \hl\, and \gpm\, for varying $c$. }}    
            \label{fig:xp1}
        \end{subfigure}
        \hfill
        \begin{subfigure}[b]{0.475\textwidth}  
            \centering 
            \includegraphics[width=\textwidth]{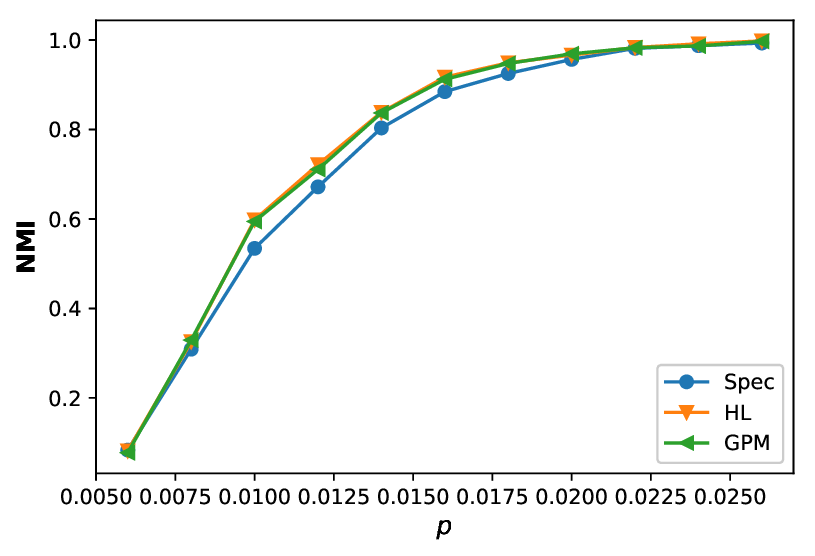}
            \caption[]%
            {{\small Relative performance of \spec, \hl\, and \gpm\, for varying $p$.}}    
            \label{fig:xp2}
        \end{subfigure}
        \vskip\baselineskip
        \begin{subfigure}[b]{0.475\textwidth}   
            \centering 
            \includegraphics[width=\textwidth]{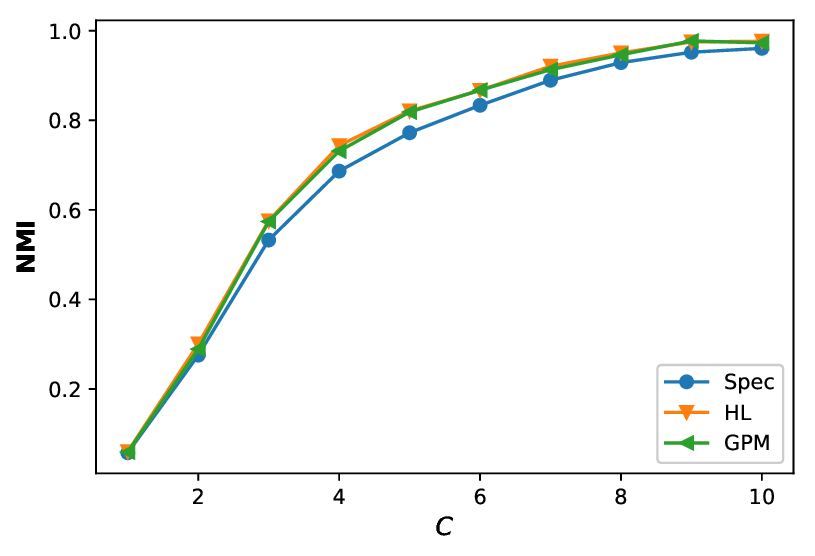}
            \caption[]%
            {{\small Relative performance of \spec, \hl\, and \gpm\, for varying $C$.}}    
            \label{fig:xp3}
        \end{subfigure}
        \hfill
        \begin{subfigure}[b]{0.475\textwidth}   
            \centering 
            \includegraphics[width=\textwidth]{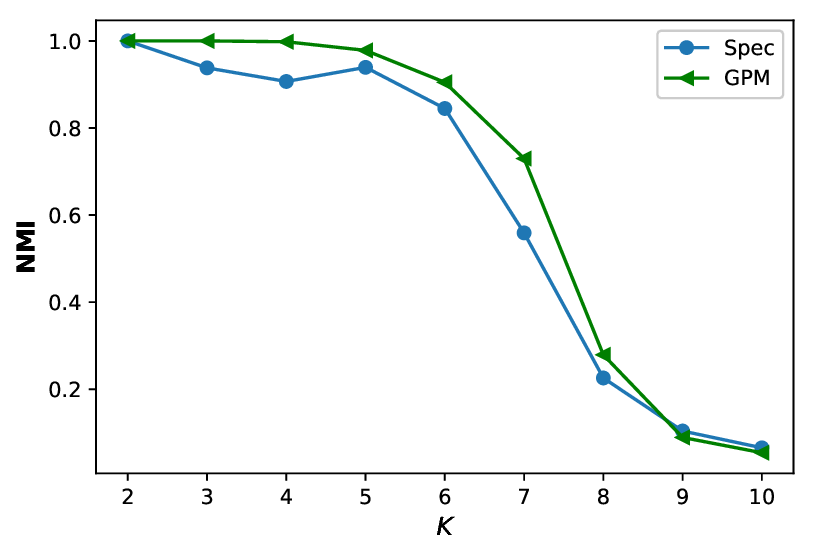}
            \caption[]%
            {{\small Relative performance of \spec\, and \gpm\, for varying $K$.}}    
            \label{fig:xp4}
        \end{subfigure}
        \caption[ ]
        {\small Relative performance of \spec, \hl\, and \gpm\, for varying parameters} 
        \label{fig:mean and std of nets}
    \end{figure*}

\paragraph{Case $K=L\neq 2$.}
We fix $n_1=1000, n_2=10000, p=0.05$ and $c=0.5$ and vary $K$ from $2$ to $10$. As can be seen from Figure \ref{fig:xp4}, the performance of \spec\, decreases faster than \gpm\, when $K$ increases.

\rev{For these experiments, we assumed that the number of communities $K$ was known. In practice, one can try to estimate the number of clusters by using the simple and popular elbow method which chooses a $K$ that maximizes the spectral gap $\lambda_k(B)- \lambda_{k+1}(B)$ (recall $B$ from Algorithm \ref{alg:spec}). 
%Alternatively, We could also derive similarly to SBM information criteria as the BIC, ICL, etc. 
From a more theoretical point of view, it would be interesting to see if the selection method proposed by \cite{zheng20}, which also comes with optimality guarantees, can be extended to the BiSBM setting.}
\section{Conclusion and future work}
In this work, we proposed an algorithm based on the GPM to cluster the rows of bipartite graphs in the high-dimensional regime where $n_2\gg n_1\log n_1$. We analyzed our algorithm under a relatively general BiSBM, that incorporates as a special case the model with $K=L=2$ studied recently by \cite{Ndaoud2021ImprovedCA}. When specialized to the SBiSBM with $K=L=2$ communities, our rate of convergence matches the one obtained by \cite{Ndaoud2021ImprovedCA}. We also extend the aforementioned  work in another direction by showing that the derived rate of convergence of the GPM for the SiSBM (with $K=L=2$)  is minimax optimal.

We only considered binary bipartite graphs in this work, but our algorithm could possibly be extended to the weighted case. Some heterogeneity amongst blocks could also be introduced, similar to the Degree-Corrected SBM. Another interesting direction would be to improve the theoretical guarantees obtained for the spectral method as we observed that it has very good performance in practice. \rev{A natural way do that would be to adapt the entrywise eigenvector techniques developed by \cite{abbe20} to the BiSBM setting.}

%\hemant{Explain more: interesting to use leave one out technique of Abbe to analyze spectral method; it has been used for analyzing SBM etc (Strohmer for e.g.)}

%Some future work direction : weighted bipartite graph, tight bounds to explain the performances of spectral methods, integrate side information, etc.
%\newpage
\bibliography{references}
\clearpage
\appendix
\begin{center}
\Large \textbf{Supplementary Material}
\end{center}

\section{Useful inequalities}
\begin{lemma}\label{lem:ht}
For all $z, z' \in [K]^{n_1}$ we have \[ h(z,z') \leq \frac{l(z,z')}{\Delta_{min}^2}\] where $h(z,z')= \sum_{i\in [n_1]}\indic_{z_i \neq z'_i} $ is the Hamming distance between $z$ and $z'$.
\end{lemma}

\begin{proof}
For all $z,z'$ we have \[ \sum_{i \in [n_1]} \indic_{z_i \neq z'_i} \leq \sum_{i \in [n_1]} \frac{\Delta(z_i, z_i')^2}{\Delta_{min}^2}\indic_{z_i \neq z_i'} =  \frac{l(z,z')}{\Delta_{min}^2}. \]
\end{proof}

%\gb{We need to fix the notation for the size of communities.}
\begin{lemma}\label{lem:nt} Assume that for some $\alpha \geq 1$ we have \[  \frac{n_1}{\alpha K} \leq |\calC_k| \leq \frac{\alpha n_1}{K} \text{ for all } k\in [K].\] 
Let $\calC_k^{(t)}=\lbrace i\in [n_1]: z^{(t)}=k\rbrace$ be the nodes associated with the community $k$ at some iteration $t \geq 0$, with $z^{(t)} \in [K]^{n_1}$. 
If $l(z,z^{(t)}) \leq  n_1\Delta_{min}^2/(2\alpha K)$ then for all $k \in [K]$,
\[  \frac{ n_1}{2\alpha K} \leq |\calC_k^{(t)}| \leq \frac{2\alpha n_1}{K}.\]
\end{lemma}
\begin{proof}
This is a standard result, we recall the proof for completeness. Since for all $k \in [K]$ we have $n_1/(\alpha  K) \leq |\calC_k| \leq \alpha n_1/K$, \begin{align*}
    \sum_{i \in \calC_k^{(t)}}1 \geq \sum_{i \in \calC_k \cap \calC_k^{(t)}}1 &\geq \sum_{i \in \calC_k}1-\sum_{i \in [n_1]}\indic_{z_i \neq z_i^{(t)}}\\
    & \geq  \frac{n_1}{\alpha K} - h(z,z^{(t)}) 
    \\ & \overset{\text{Lemma }\ref{lem:ht}}{\geq } \alpha \frac{n_1}{K} - \frac{l(z,z^{(t)})}{\Delta_{min}^2}\\
    &\geq  \frac{\alpha n_1}{2K}
\end{align*}
by assumption. The other inequality can be proved in a similar way.
\end{proof}

\begin{lemma}\label{lem:wt}Assume that $A \sim BiSBM(Z_1,Z_2,\Pi)$ with approximately balanced communities.  Then for all $t \geq 0$ and $z^{(t)}$ such that $\lt \leq  n_1\Delta_{min}^2/(2\alpha K)$, the following holds.
\begin{enumerate}
    \item $ \max_{k\in [K]}|| W^{(t)}_{:k}-W_{:k}|| \lesssim \frac{(\alpha K)^{1.5}}{n_1^{1.5}\Delta_{min}^2} l(z^{(t)},z),$
    \item $||Z_1^\top W^{(t)}|| \lesssim 1 $.
\end{enumerate}
\end{lemma}
%
%

%\hemant{It would be better to provide the proof outline for the convenience of the reader. Also, the above statement says ``equal size'' communities but this should be for approximately equal sizes? Finally, rather than saying that ``the conditions of Theorem 2 are satisfied'' it will be cleaner to explicitly state what conditions you require here. This makes the Lemma more modular and the presentation cleaner (Another reason: we will likely have to edit the conditions of Theorem 2 later, so its better to have self contained lemmas in order to avoid going back and forth).}

\begin{proof} The proof is an adaptation of Lemma $4$ in \cite{Han2020ExactCI}, see also Lemma $14$ in \cite{braun2021iterative}. Since we consider a slightly different setting, we outline the details for completeness.

Let $n_{min}=\min_k |\calC_k|$. First observe that $Z_1$ is rank $K$ and $\lambda_K(Z)= \sqrt{n_{min}}$ so that 
\[ || W^{(t)}_{:k}-W_{:k}|| \leq || W^{(t)}-W|| \leq n_{min}^{-1/2} ||I-Z_1^\top W^{(t)} ||_F.\] For any $k \in [K]$, denote $\delta_k=1-(Z_1^\top W^{(t)})_{kk}$. Since for all $k, k' \in [K]$ \[ (Z_1^\top W^{(t)})_{kk'}=\frac{\sum_{i \in \calC_k}\indic_{z_i^{(t)}=k'}}{|\calC_{k'}^{(t)}|},\] we have \[ 0 \leq \delta_k \leq 1, \quad \sum_{k'\in [K]\backslash k} (Z_1^\top W^{(t)})_{k'k} = \delta_k .\] Therefore, 
\begin{align*}
    ||Z_1^\top W^{(t)}-I ||_F & = \sqrt{\sum_{k \in [K]} \left( \delta_k^2+ \sum_{k'\in [K]\backslash k} (Z_1^\top W^{(t)})_{k'k}^2 \right)}\\
    & \leq \sqrt{\sum_{k \in [K]} \left( \delta_k^2+ \left(\sum_{k'\in [K]\backslash k} (Z_1^\top W^{(t)})_{k'k}\right)^2 \right)}\\
    & \leq \sqrt{2\sum_{k \in [K]} \delta_k^2} \leq \sqrt{2}\sum_{k \in [K]}\delta_k \\
    & = \sqrt{2} \sum_{k\in [K]}\frac{\sum_{i \in \calC^{(t)}_k} \indic_{z_i \neq k}}{|\calC_{k'}^{(t)}|}\\
    & \leq \sqrt{2}\max_k(|\calC_{k'}^{(t)}|)^{-1} \sum_{i \in [n]}\indic_{z_i \neq z_i^{(t)}}\\
    & \overset{\text{Lemma \ref{lem:nt}}}{\lesssim } \frac{K\alpha}{n_1}h(z,z^{(t)}) \overset{\text{Lemma \ref{lem:ht}} }{\lesssim }  K\alpha \frac{l(z,z^{(t)})}{n_1 \Delta_{min}^2} \numberthis \label{eq:izw}.
\end{align*}
%\gb{We can make the dependency in $\alpha$ appear.}

For the second inequality, observe that since $Z_1^\top W =I_K$ we have \begin{align*}
        ||Z_1^\top W^{(t)}|| &\leq 1 + || Z_1^\top (W^{(t)}-W)||\\
        & = 1 + ||I - Z_1^\top W^{(t)}||\\
        & \lesssim 1+\alpha K\frac{\lt}{n_1\Delta^2_{min}} \tag{by Equation \eqref{eq:izw}}\\
        &\lesssim 1 \tag{by assumption on $\lt$}.
    \end{align*}
\end{proof}

\section{Concentration inequalities}
\subsection{General results}
%\gb{There is a mistake in the concentration inequality 2 because I forgot an important log factor...}
\begin{lemma}
\label{lem:conc_mat} Assume that $A \sim BiSBM(Z_1,Z_2,\Pi)$, and recall 
%$n_2 \gg n_1\log n_1$ and $n_1n_2p_{max}^2\gtrsim KL\log n_1$. 
$E=A-\expec(A)$. For $n_1$ large enough, the following holds true.
%Then with probability at least $1-n_1^{-\Omega(1)}$ the following holds.
\begin{enumerate}
    \item $\prob(||E||\lesssim \sqrt{n_2p_{max}}) \geq 1-n_1^{-\Omega(1)}$ if $n_2 p_{\max} \gtrsim \log n_2$ and $n_2 \geq n_1$.
    
    \item $\prob(||EE^\top -\expec(EE^\top)|| \lesssim   \max(\log n_1, \sqrt{n_1n_2}p_{max})) \geq 1-n_1^{-\Omega(1)} $. Moreover, this event implies $\norm{\calH(EE^\top)} \lesssim \max(\log n_1, \sqrt{n_1n_2}p_{max}))$.
    
    \item $\prob(||EZ_2|| \lesssim \sqrt{\rev{\alpha}n_1n_2p_{max}/L}) \geq 1-n_1^{-\Omega(1)}$ if $\rev{\alpha}n_2 p_{\max} \gtrsim L \log n_1$ and $\sqrt{\frac{\rev{\alpha} n_1 n_2 p_{\max}}{L}} \gtrsim \log n_1$.
\end{enumerate}
\end{lemma}

\begin{proof}
\begin{enumerate}
     \item The first inequality follows from classical proof techniques. We first convert $E$ into a square symmetric matrix as follows 
     \[\tilde{E}=
     \begin{pmatrix}
     0 & E\\
     E^\top &0
     \end{pmatrix}.
     \]
     It is easy to verify that $\norm{\tilde{E}}=\norm{E}$. Since $\tilde{E}$ is a square symmetric matrix with independent entries, we can use the result of Remark 3.13 in \cite{bandeira2016} and obtain with probability at least $1-n_1^{-\Omega(1)}$ that  \[ \norm{\tilde{E}}\lesssim \sqrt{n_2p_{max}} \] since $n_2p_{max}\gtrsim \log(n_1+ n_2)$. This condition is ensured since $n_2 \geq n_1$ and $n_2 p_{\max} \gtrsim \log n_2$.
     
     \item The second inequality follows from the recent work of \cite{Cai2020OnTN}. First observe that
    \[ ||\calH(EE^\top)|| \leq 2|| EE^\top -\expec(EE^\top)||\]
    since $||\calH(A)||\leq ||A||+||\diag(A)|| \leq 2||A||$ for all square matrices $A$ and $\calH(EE^\top) = \calH(EE^\top -\expec(EE^\top))$. Theorem 4 in \cite{Cai2020OnTN} shows that 
    \[ \expec (|| EE^\top -\expec(EE^\top)||)\lesssim \max(\sqrt{n_1n_2} p_{max}, \log n_1).\]
    It is however more difficult to get a high probability bound from this last inequality since we can no longer use Talagrand's inequality as in \cite{bandeira2016}. However, we can use  the moments to obtain a tail bound as in Theorem $5$ in \cite{Cai2020OnTN}. This theorem is stated for  matrices with Gaussian entries, but if instead of Lemma $1$ we use Lemma $9$ of \cite{Cai2020OnTN}, we obtain a similar result for bounded sub-Gaussian entries. Since the variance parameters $\sigma_R$ and $\sigma_C$ that appear in the statement of Theorem $5$ in \cite{Cai2020OnTN} satisfy $\sigma_R^2\leq n_1p_{max} $ and $\sigma_C^2\leq n_2p_{max} $ we obtain the result.
    
    \item The last inequality can be obtained by using the proof techniques in \cite{bandeira2016} as follows. In order to extend the concentration result from a matrix $Y$ with independent standard Gaussian entries to a matrix $X$ with symmetric sub-Gaussian entries, the key is to upper bound all the moments of $X_{ij}$ by moment of $Y_{ij}$. This can be done by using the boundedness of $Z_{ij}$ as in Corollary $3.2$, or the sub-Gaussian norm of $X_{ij}$ as in Corollary $3.3$ of \cite{bandeira2016}. But in our case, none of these bounds gives a good result. However, the proof of Theorem $1.1$ (and its extensions) only requires control of the moments of the order $\log n_1$.  For a Binomial r.v. $X$ with parameters $\rev{\alpha} n_2/L$ and $p_{max}$, we have, according to Theorem 1 in \cite{AHLE2022109306}, for all $c\in \mathbb{N} ^*$, 
    \[ \expec(X^c)\leq (\rev{\alpha} n_2p_{max}/L)^c e^{c^2/(2\rev{\alpha} n_2p_{max}/L)}. \]
    Let $X'$ be an independent copy of $X$. Since $\rev{\alpha} n_2p_{max}/L\gtrsim \log n_1$ by assumption, $e^{c^2/(2\rev{\alpha}n_2p_{max}/L)}\leq e^{\gamma c}$ for $c \asymp \log (n_1)$ and an absolute constant $\gamma>0$. Hence
    \[ \expec((X-X')^c)\leq \sum_{i\leq c}\binom{c}{i}\expec(X^i)\expec(X^{c-i}) \leq 2^c(\rev{\alpha} n_2p_{max}/L)^c e^{\gamma c}.\]
    Observe that $(2e^\gamma)^c\lesssim \expec Y_{ij}^{2c} = O((2c)^{c})$ for every $c$, so that for all even $c$ we have
    \[ \expec\left(\frac{L(X-X')}{2e^\gamma \rev{\alpha} n_2p_{max}}\right)^{c} \lesssim \expec Y_{ij}^{2c}. \]
    %
   %\hemant{I think there is some polishing needed in some statements above, especially with regards to $c$...}
    We can now use the same argument as in Corollary $3.2$ of \cite{bandeira2016} to conclude that the matrix $M$ with independent entries generated with the same law as $X-X'$ satisfies with probability at least $1-O(n_1^{\Omega(1)})$ (since $\sqrt{\rev{\alpha}n_1n_2p_{max}/L} \gtrsim \log n_1$ by assumption)
    \[ \norm{M} \leq \sqrt{\rev{\alpha}n_1n_2p_{max}/L} .\]
    When the random variables are only centered, we can use the symmetrization argument of Corollary $3.3$ to finally obtain
    \[ || EZ_2^\top || \lesssim \sqrt{\rev{\alpha}n_1n_2p_{max}/L},\]
    with probability at least $1-O(n_1^{\Omega(1)})$. 
\end{enumerate}
\end{proof}

\begin{remark} The second concentration inequality of the above lemma slightly improves Proposition 1  and Theorem 4 in \cite{Ndaoud2021ImprovedCA}. The third concentration could be of independent interest and be applied  for example in multilayer network analysis where matrices with independent Binomial entries arise naturally as the sum of the adjacency matrices of the layers, see e.g. \cite{ paul_chen,Braun2021ClusteringMG}.
%\gb{see my work ? Paul and Chen ?} \hemant{Yes I think so}. 
\end{remark}

\subsection{Control of the oracle error term}

\begin{lemma}\label{lem:oracle} Assume that the assumptions of Theorem \ref{thm:sbisbm} hold.  Recall that \[\ov = \left\lbrace  \tilde{E}_{i:}(W_{:z_i}-W_{:k}) \leq -(1-\delta)\Delta^2(z_i,k) \right\rbrace .\] We have for all $0<\delta\leq \frac{1}{4\eta \alpha}$, $k\neq z_i\in [K]$ and $i\in [n_1]$
 \[ \prob(\ov) \leq  e^{-\frac{ \beta^2 }{12eL\alpha^3} \frac{n_1n_2p_{max}^2}{KL}}=e^{-\tilde{\Delta}^2}.\]
\end{lemma}

\begin{proof}
%Recall that \[\Omega_1 = \left\lbrace  \tilde{E}_{i:}(W_{:z_i}-W_{:k}) \leq -(1-\delta)\Delta^2(z_i,k) \right\rbrace .\] 
The event $\ov$ holds if and only if \[ B_{i:}(W_{:k}-W_{:z_i}) \geq -\delta \Delta^2(z_i,k).\] But we can decompose this quantity as \begin{align*}
     B_{i:}(W_{:k}-W_{:z_i})&= \sum_{j\neq i} B_{ij}(W_{jk}-W_{jz_i})\\
    &= |\calC_k|^{-1}\sum_{j\in \calC_k\setminus \lbrace i \rbrace }\langle A_{i:},A_{j:} \rangle-|\calC_{z_i}|^{-1}\sum_{j\in \calC_{z_i}\setminus \lbrace i \rbrace }\langle A_{i:},A_{j:} \rangle\\
    &=\langle A_{i:}, \tilde{A}_k -\tilde{A}_{z_i} \rangle
\end{align*}
 where $\tilde{A}_k =  \frac{1}{|\calC_k|}\sum_{j\in \calC_k  } A_{j:}$ and $\tilde{A}_{z_i} = \frac{1}{|\calC_{z_i}|} \sum_{j\in \calC_{z_i}\setminus \lbrace i \rbrace  } A_{j:}$ are independent random variables. Since the index $i$ doesn't appear in the second sum, $\tilde{A}_{z_i}$ is independent from $A_{i:}$. By definition the entries of $\tilde{A}_k$ and $\tilde{A}_{z_i}$ are independent normalized binomial random variables whose parameters vary depending on the community associated with the entry.

%  Toward this end, let us further decompose $\langle A_{i:}, \tilde{A}_k -\tilde{A}_{z_i} \rangle$ as \[ \underbrace{\sum_{l\in \calC'_{z_i}} A_{il}(\tilde{A}_{kl}-\tilde{A}_{z_il})}_{M_1}+\underbrace{\sum_{l\in \calC'_{k}} A_{il}(\tilde{A}_{kl}-\tilde{A}_{z_il})}_{M_2} + \underbrace{\sum_{k'\in [K]\setminus{\lbrace k,z_i \rbrace }}\sum_{l\in \calC'_{k'}} A_{il}(\tilde{A}_{kl}-\tilde{A}_{z_il})}_{M_3}.\]
 
We are now going to bound the moment generating function of $M:=\langle A_{i:}, \tilde{A}_k -\tilde{A}_{z_i} \rangle$, conditionally on $A_{i:}$. Observe that $M$ is a sum of independent random variables. Recall that if $X\overset{}{\sim} \mathcal{B}(p)$ then $\expec(e^{tX})=(e^tp+1-p)$. Hence, for $t>0$, conditionally on $A_{ij}$, each summand has a m.g.f equal to \begin{align*}
 \log \expec (e^{tA_{ij}(\tilde{A}_{kj}-\tilde{A}_{z_ij})}|A_{ij}) &= |\calC_{k}|\log(e^{\frac{tA_{ij}}{|\calC_{k}|}}\Pi_{k\zz_{j}}+1-\Pi_{k\zz_{j}})+(|\calC_{z_i}|-1)\log(e^{-\frac{tA_{ij}}{|\calC_{z_i}|}}\Pi_{z_i\zz_{j}}+1-\Pi_{z_i\zz_{j}}) \\
&\leq |\calC_{k}|\Pi_{k\zz_{j}}(e^{\frac{tA_{ij}}{|\calC_{k}|}}-1)+(|\calC_{z_i}|-1)\Pi_{z_i\zz_{j}}(e^{-\frac{tA_{ij}}{|\calC_{z_i}|}}-1)    
\end{align*} by using the fact that $\log(1+x) \leq x$ for all $x>-1$. 

Fix $t=t^*=\epsilon \frac{n_1}{\alpha K}$ for a parameter $\epsilon \in (0,1)$ that will be fixed later. We have by Taylor Lagrange formula for all $t\leq \frac{n_1}{\alpha K}$  \[e^{\frac{t}{|\calC_{k}|}}-1 \leq \frac{t}{|\calC_{k}|}+\frac{e}{2}\left(\frac{t}{|\calC_{k}|}\right)^2\] and  \[ e^{\frac{-t}{|\calC_{z_i}|}}-1 \leq -\frac{t}{|\calC_{z_i}|}+\frac{e}{2}\left(\frac{t}{|\calC_{z_i}|}\right)^2 .\]
By using these upper bounds we get \[ \log \expec (e^{t^*A_{ij}(\tilde{A}_{kl}-\tilde{A}_{z_ij})}|A_{ij}) \leq \left(\Pi_{k\zz_{j}}-\frac{|\calC_{z_i}|-1}{|\calC_{z_i}|}\Pi_{z_i\zz_{j}}\right)t^*A_{ij}+\frac{e}{2}(t^*)^2A_{ij}\left( \frac{\Pi_{k\zz_{j}}}{|\calC_k|}+\frac{\Pi_{z_i\zz_{j}}}{|\calC_{z_i}|}\right). \]
Hence by using independence we obtain \[ \log \expec(e^{t^*M}|A_{i:})\leq \sum_{j\in [n_2]}\left[ \underbrace{\left(\Pi_{k\zz_{j}}-\frac{|\calC_{z_i}|-1}{|\calC_{z_i}|}\Pi_{z_i\zz_{j}}\right)t^*+\frac{e}{2}(t^*)^2\left( \frac{\Pi_{k\zz_{j}}}{|\calC_k|}+\frac{\Pi_{z_i\zz_{j}}}{|\calC_{z_i}|}\right)}_{t_{ij}}\right]A_{ij}.\]
Using Markov inequality  and  the fact that $\Delta^2(z_i,k)\leq \eta \beta \alpha n_2p_{max}^2/L$ by  \eqref{eq:delta_zik} leads to \[\rev{\expec }\left(\prob(\ov|A_{i:})\right)\leq e^{\delta \epsilon \frac{n_1}{\alpha K}\Delta^2(z_i,k)}\rev{\expec}\left(\expec(e^{t^*M}|A_{i:})\right) \leq e^{\delta \epsilon \frac{\beta n_1n_2p_{max}^2}{ LK}} \prod_j\expec(e^{t_{ij}A_{ij}}).\]
But since $A_{ij}$ is a Bernoulli random variable with parameter $\Pi_{z_iz_{2j}}$ we have \[ \expec(e^{t_{ij}A_{ij}}) = (e^{t_{ij}}-1)\Pi_{z_iz_{2j}}+1.\]
By our choice of $t^*$, $t_{ij}=O(n_1p_{max})=o(1)$ so that \[ \prod_j\left((e^{t_{ij}}-1)\Pi_{z_iz_{2j}}+1\right)\leq e^{\sum_j (e^{t_{ij}}-1)\Pi_{z_iz_{2j}}}.\] Here we use the fact that for $x_1, \ldots, x_n>-1$ we have $\prod_{i\in [n]} (1+x_i)\leq e^{\sum_{i\in [n]}x_i}$.
But again, by using Taylor Lagrange formula, we have $e^{t_{ij}}-1=t_{ij}+O(t_{ij}^2)$. Consequently \[ \prod_j\expec(e^{t_{ij}A_{ij}}) \leq e^{\sum_j (t_{ij}+o(t_{ij}))\Pi_{z_iz_{2j}}}.\]
We can write $\sum_j t_{ij}\Pi_{z_iz_{2j}}$ as \[ t^* \underbrace{\sum_{l\in [L]} |\calC_{l}'|\left( \Pi_{z_il}\Pi_{kl} - \Pi_{z_il}^2\left(\frac{|\calC_{z_i}|-1}{|\calC_{z_i}|}\right)\right)}_{A1}+\frac{e(t^*)^2}{2}\underbrace{\sum_{l\in [L]}|\calC_{l}'|\left( \frac{ \Pi_{z_il}\Pi_{kl}}{|\calC_k|}+ \Pi_{z_il}^2\left(\frac{|\calC_{z_i}|-1}{|\calC_{z_i}|^2}\right)\right)}_{A_2}.\]
By Assumption \ref{ass:balanced_part} and \ref{ass:diag_dom}  we have \[ -A_1 \geq \frac{n_2}{\alpha L}\left(\beta p_{max}^2 + o(p_{max}^2)\right)\] and
\[ A_2 \leq \frac{\alpha K}{n_1}n_2 p_{max}^2.\]
Let $\epsilon =  \frac{ \beta}{3eL\alpha}$ (recall that $t^*=\epsilon \frac{n_1}{\alpha K}$). By this choice of $\epsilon$ we have \[ \frac{e\epsilon n_1}{2\alpha K}A_2 \leq \frac{|A_1|}{2}.\]

Consequently  we have for all $\delta \leq 1/(4\eta \alpha)$
\[ \prob(\ov)\leq e^{-\frac{\epsilon \beta }{2\alpha^2} \frac{n_1n_2p_{max}^2}{KL}+\delta \epsilon \frac{n_1}{\alpha K}\Delta^2(z_i,k)} \leq e^{-\frac{\epsilon \beta }{4\alpha^2} \frac{n_1n_2p_{max}^2}{KL}}\leq  e^{-\frac{ \beta^2 }{12eL\alpha^3} \frac{n_1n_2p_{max}^2}{KL}}.\] 
\end{proof}

\subsection{Proof of Lemma \ref{lem:cond_prob_minimax}}
\label{app:proof_lem_1}

The first inequality is a direct consequence of Hoeffding concentration inequality. For the second inequality, observe that  \[ \prob_{A_{-i:}|z_{-i}}(\calE_2)=\expec_{\zz}\prob_{A_{-i:}|z_{-i},\zz}(\calE_2) \] and 
\[\prob_{A_{-i:}|z_{-i},\zz}\left(\sum_{i'\in \calC_+}A_{i'j}-\sum_{i'\in \calC_-}A_{i'j}= 0\right)\geq \prob (A_{i'j}=0 \text{ for all }i')\gtrsim (1-p)^{n_1}=1-n_1p.\]
 To obtain an upper bound on this probability, we can use Paley-Zigmund inequality as follows. Let us denote $Z=|\sum_{i'\in \calC_+}A_{i'j}-\sum_{i'\in \calC_-}A_{i'j}|$. We then have \begin{align*}
     \prob_{A_{-i:}|z_{-i},\zz} (Z=0)=1-\prob_{A_{-i:}|z_{-i},\zz}(Z>0)&\leq 1- \prob_{A_{-i:}|z_{-i},\zz}(Z>\theta \expec Z)\\ &\leq 1-(1-\theta)^2\frac{(\expec_{A_{-i:}|z_{-i},\zz}( Z))^2}{\expec_{A_{-i:}|z_{-i},\zz} (Z^2)}
 \end{align*}
 for $\theta \in (0,1)$ by Paley-Zigmund inequality. It is easy to check that \[\expec_{A_{-i:}|z_{-i},\zz} Z\gtrsim n_1(p-q) \text{ and } \expec_{A_{-i:}|z_{-i},\zz}(Z^2)\asymp n_1p.\]  Consequently, $\prob_{A_{-i:}|z_{-i},\zz}(Z=0)= 1-\Theta(n_1p)$. By using independence over $j$ and Chernoff's multiplicative bound we obtain that \[ \prob_{A_{-i:}|z_{-i},\zz}\left(\sum_j \indic_{\lbrace \sum_{i'\in \calC_+}A_{i'j}-\sum_{i'\in \calC_-}A_{i'j}\neq 0 \rbrace} = \Theta( n_2n_1p)\right) \geq 1-e^{-cn_2n_1p}\] for any realization of $z_{-i}$ and $z'$. Hence the stated bound follows.
 
 Finally it remains to control the probability of $\calE_3$. Under $\tilde{\prob}_{A_{i:}|A_{-i:},z}$, for each $j\in J_b$, $A_{ij}$ are independent and distributed as a mixture of Bernoulli of parameters $p$ and $q$. The size of the set $J_b$ is also controlled by the assumption $A_{-i:}\in \calE_2$ (see equation \eqref{eq:jb}). Since under $\tilde{\prob}_{A_{i:}|A_{-i:},z}$ for all $j\in J_b$, $A_{ij}$ are independent r.v. with density $g$, we obtain by stochastic domination (we can replace a mixture of Bernoulli with parameters $p$ and $q$ by a Bernoulli of parameter $p$) and Chernoff's bound that $\tilde{\prob}_{A_{i:}|A_{-i:},z}(\calE_3)\geq 1-e^{-\Theta(n_2p)}$.

\end{document}